\documentclass[a4paper,10pt]{amsart}
\usepackage{amsmath,amsthm,amssymb,latexsym,enumerate,xcolor,hyperref}
\usepackage{graphicx}
\usepackage{eucal}
\usepackage[numbers]{natbib}

\numberwithin{equation}{section}

\begin{document}

\newtheorem{thm}{Theorem}[section]
\newtheorem{prop}[thm]{Proposition}
\newtheorem{lem}[thm]{Lemma}
\newtheorem{cor}[thm]{Corollary}
\theoremstyle{remark}
\newtheorem{rem}[thm]{\bf Remark}
\newtheorem{ex}[thm]{\bf Example}
\newtheorem*{defn}{Definition}

\newcommand{\DD}{\mathbb{D}}
\newcommand{\NN}{\mathbb{N}}
\newcommand{\ZZ}{\mathbb{Z}}
\newcommand{\QQ}{\mathbb{Q}}
\newcommand{\RR}{\mathbb{R}}
\newcommand{\CC}{\mathbb{C}}
\renewcommand{\SS}{\mathbb{S}}

\renewcommand{\theequation}{\arabic{section}.\arabic{equation}}

\newcommand{\supp}{\mathop{\mathrm{supp}}}    
\newcommand{\sgn}{\mathrm{sgn}}    

\newcommand{\re}{\mathop{\mathrm{Re}}}   
\newcommand{\im}{\mathop{\mathrm{Im}}}   
\newcommand{\dist}{\mathop{\mathrm{dist}}}  
\newcommand{\link}{\mathop{\circ\kern-.35em -}}
\newcommand{\spn}{\mathop{\mathrm{span}}}   
\newcommand{\ind}{\mathop{\mathrm{ind}}}   
\newcommand{\rank}{\mathop{\mathrm{rank}}}   
\newcommand{\Fix}{\mathop{\mathrm{Fix}}}   
\newcommand{\codim}{\mathop{\mathrm{codim}}}   
\newcommand{\conv}{\mathop{\mathrm{conv}}}   

\newcommand{\eps}{\varepsilon}
\newcommand{\ep}{\epsilon}

\newcommand{\cl}{\overline}
\newcommand{\pa}{\partial}
\newcommand{\ve}{\varepsilon}
\newcommand{\zi}{\zeta}
\newcommand{\Si}{\Sigma}
\newcommand{\one}{{\mathcal X}}
\newcommand{\cA}{{\mathcal A}}
\newcommand{\cD}{{\mathcal D}}
\newcommand{\cE}{{\mathcal E}}
\newcommand{\cG}{{\mathcal G}}
\newcommand{\cH}{{\mathcal H}}
\newcommand{\cI}{{\mathcal I}}
\newcommand{\cJ}{{\mathcal J}}
\newcommand{\cK}{{\mathcal K}}
\newcommand{\cL}{{\mathcal L}}
\newcommand{\cN}{{\mathcal N}}
\newcommand{\cR}{{\mathcal R}}
\newcommand{\cS}{{\mathcal S}}
\newcommand{\cT}{{\mathcal T}}
\newcommand{\cU}{{\mathcal U}}
\newcommand{\OM}{\Omega}
\newcommand{\B}{\bullet}
\newcommand{\un}{\underline}
\newcommand{\ol}{\overline}
\newcommand{\ul}{\underline}
\newcommand{\vp}{\varphi}
\newcommand{\AC}{\mathop{\mathrm{AC}}}   
\newcommand{\Lip}{\mathop{\mathrm{Lip}}}   
\newcommand{\es}{\mathop{\mathrm{esssup}}}   
\newcommand{\les}{\mathop{\mathrm{les}}}   
\newcommand{\nid}{\noindent}
\newcommand{\pzr}{\phi^0_R}
\newcommand{\pir}{\phi^\infty_R}
\newcommand{\psr}{\phi^*_R}
\newcommand{\pow}{\frac{N}{N-1}}
\newcommand{\ncl}{\mathop{\mathrm{nc-lim}}}   
\newcommand{\nvl}{\mathop{\mathrm{nv-lim}}}  
\newcommand{\la}{\lambda}
\newcommand{\La}{\Lambda}    
\newcommand{\de}{\delta}    
\newcommand{\fhi}{\varphi} 
\newcommand{\ga}{\gamma}    
\newcommand{\ka}{\kappa}   

\newcommand{\core}{\heartsuit}
\newcommand{\diam}{\mathrm{diam}}

\newcommand{\lan}{\langle}
\newcommand{\ran}{\rangle}
\newcommand{\loc} {\mathrm{loc}}
\newcommand{\tr}{\mathop{\mathrm{tr}}}
\newcommand{\diag}{\mathop{\mathrm{diag}}}
\newcommand{\dv}{\mathop{\mathrm{div}}}

\newcommand{\al}{\alpha}
\newcommand{\be}{\beta}
\newcommand{\Om}{\Omega}
\newcommand{\na}{\nabla}

\newcommand{\cC}{\mathcal{C}}
\newcommand{\cM}{\mathcal{M}}
\newcommand{\nr}{\Vert}
\newcommand{\De}{\Delta}
\newcommand{\cX}{\mathcal{X}}
\newcommand{\cP}{\mathcal{P}}
\newcommand{\om}{\omega}
\newcommand{\si}{\sigma}
\newcommand{\te}{\theta}
\newcommand{\Ga}{\Gamma}

\title[{The location of hot spots and other extremal points}]{The location of hot spots \\ and other extremal points}

\author[Rolando Magnanini]{Rolando Magnanini}
\address{Dipartimento di Matematica ``U. Dini'', Universit\`a a di Firenze, viale Morgagni 67/A, 50134 Firenze, Italy}
\email{magnanini@unifi.it}
\urladdr{http://web.math.unifi.it/users/magnanin}

\author{Giorgio Poggesi}
\address{Department of Mathematics and Statistics, The University of Western Australia, 35 Stirling Highway, Crawley, Perth, WA 6009, Australia}
    \email{giorgio.poggesi@uwa.edu.au}

\begin{abstract}
In a domain of the Euclidean space, we estimate from below the distance to the boundary of global maximum points of solutions of elliptic and parabolic equations with homogeneous Dirichlet boundary values. As reference cases, we first consider the torsional rigidity function of a bar, the first mode of a vibrating membrane, and the temperature of a heat conductor grounded to zero at the boundary. 
Our main results are presented for domains with a mean convex boundary and compare that distance to the inradius of the relevant domain. 
\par
For the torsional rigidity function, the obtained bound only depends on the space dimension. The more general case of a boundary which is not mean convex is also considered. However, in this case the estimates also depend on some geometrical quantities such as the diameter and the radius of the largest exterior osculating ball to the relevant domain, or the minimum of the mean curvature of the boundary.
\par
Also in the case of the first mode, the relevant bound only depends on the space dimension. Moreover, it largely improves on an earlier estimate obtained for convex domains by the first author and co-authors. The bound related to the temperature depends on time and the initial distribution of temperature. Such a bound is substantially consistent with what one obtains in the stationary situation.
\par
The methods employed are based on elementary arguments and existing literature, and can be extended to other situations that entail quasilinear equations, isotropic and anisotropic, and also certain classes of semilinear equations. 
\end{abstract}

\maketitle

\raggedbottom

\section{Introduction}
Any Calculus student is aware of the importance of the critical (or stationary) points of a differentiable function $u$ for describing its graph or level surfaces. Also, from the point of view of mathematical physics, we can often interpret a function $u$ as a gravitational, electrostatic or velocity potential, or the temperature distribution in a thermal conductor, and regard its gradient $\na u$  as an underlying field of force or flow. Thus, the
critical points of $u$ (at which $\na u=0$) may be viewed as the positions of equilibrium for the field or the hot spots of the distribution of temperature, or yet the points associated to stream lines in the flow with maximal velocity.
\par
A priori information on the location of the extremum points, and also of the other critical points, of a differentiable function is thereby an important issue. Work on the location of critical points of complex polynomials dates back to C. F. Gauss. More in general one can consider the same problem for holomorphic or meromorphic functions, and their (harmonic) real or imaginary parts. We refer the reader to the 1950 treatise \cite{Wa} for an anthology of results in these circumstances. Moreover, as shown in \cite{Al1},  \cite{AM1}, \cite{AM2}, and gathered up in the recent surveys \cite{Sa, Ma},  some of these results can be extended to solutions of certain homogeneous elliptic equations (that are modelled on Laplace's equation) (see also \cite{DLT1}). Further extensions can be obtained even for certain degenerate linear and quasilinear equations (\cite{Al2}, \cite{ALR}, \cite{DLT1}, \cite{DLT2}).
\par
Still, it should be noticed that the critical points of the kind considered in (most of) the above listed papers are
never extremal points. In this paper, we shall consider three important reference situations in which extremal points occur. They entail problems still actively studied in the applications of partial differential equations to mathematical physics. They concern: the torsional rigidity of a long straight bar or the flow velocity of a viscous incompressible fluid in a straight pipe; the temperature distribution of a heat conductor; the first vibrating mode of a clamped membrane or the stationary distribution of temperature in a grounded heat conductor.
\par
In mathematical terms, the simplest situation has to do with the Dirichlet problem for the Poisson equation:
\begin{equation}
\label{torsion}
-\De u=N \ \mbox{ in } \ \Om, \quad u=0 \ \mbox{ on } \ \Ga.
\end{equation}
Here, $\Om$ is a bounded domain in $\RR^N$, $N\ge 2$ with boundary $\Ga$. The solution of \eqref{torsion} may have the physical meaning of the \textit{torsional rigidity density} of a long straight bar or the \textit{flow velocity} of a fluid flowing in a straight pipe, both with cross section $\Om$ (see \cite{Se}). 
Owing to this fluid dynimical interpretation of $u$, the maximum points of $u$ correspond to the stream lines in the fluid that flow with maximal velocity.
\par
It is well known that a unique solution $u\in C^0(\ol{\Om})\cap C^2(\Om)$ of \eqref{torsion} always exists if $\Ga$ is made of regular points for the Dirichlet problem. We know that $u$ is positive in $\Om$ by the strong maximum principle and,
once a Hopf boundary lemma is applicable and $\Ga$ is sufficiently regular, 
we can infer that the gradient $\na u$ of $u$ is not zero at points of $\Ga$. As a result, the critical points of $u$ must be inside $\Om$. It then makes sense to estimate (from below) the distance of the critical points of $u$ to the boundary in terms of some clearly measurable geometric parameters of $\Om$. 
In this paper, we shall derive such an estimate for the global maximum points of $u$. The following theorem yields a bound in terms of the \textit{inradius} of $\Om$, that is the radius of any largest ball contained in $\Om$. Here, $d_\Ga(x)$ denotes the distance of a point $x\in\Om$ to $\Ga$, which is defined by
$$
d_\Ga(x)=\min_{y\in\Ga}|x-y| \ \mbox{ for } \ x\in\ol{\Om}.
$$

\begin{thm}
\label{th:maximum-point-torsion}
Let $\Om$ be a bounded domain with mean convex boundary $\Ga$.  
\par
If $z\in\Om$ is any maximum point in $\Om$ of the solution of \eqref{torsion}, then we have that
\begin{equation}
\label{dist-torsion}
\frac{d_\Ga(z)}{r_\Om}\ge \frac1{\sqrt{N}}.
\end{equation}
\end{thm}
\noindent
We say that a surface (more precisely a boundary) is \textit{mean convex} if it is of class $C^2$ and its mean curvature $\cM$ (with respect to the inward normal) is non-negative. 
\par
A nice geometric corollary of \eqref{dist-torsion} reads as follows. 
Let $\Om\subset\RR^3$ be a dumbbell-shaped domain, with boundary $\Ga$ made by two spheres connected by a portion of a catenoid (and suitably smoothed out).  If the radius of the smaller sphere is less than $57\%$ of that of the larger one, then 
the maximum point(s) must be within the larger sphere (see Fig. \ref{fig:dumb-bell}). This fact is somewhat expected. However, \eqref{dist-torsion} quantitatively details it and, more importantly, it shows that this information is independent of the length of the dumbbell.
\par
The proof of Theorem \ref{th:maximum-point-torsion} will be presented in Section \ref{sec:torsion}. It is based on a pointwise bound from above for $|\na u|$, already existing in the literature (see \cite{PP1}), and one for $u$ from below.  
In the same section, we shall prove three related results. In the first one we consider the case in which $\Om$ is convex and obtain an improvement of \eqref{dist-torsion}, based on the \textit{John's ellipsoid} related to $\Om$ (see Theorem \ref{th:ellipse-bound}). In the other ones, we will remove the mean convexity assumption and obtain a bound for more general domains. In this case, either the negative part of the mean curvature of $\Ga$ (Corollary \ref{cor:domain-curvature}) or the diameter of $\Om$ and the radius of the largest exterior osculating ball to $\Ga$ (Corollary \ref{cor:domain-diameter-outradius})  come into play. 
\par
In our knowledge, work on the location of critical points of the torsional rigidity function is not present in the literature. This issue has been instead investigated for the \textit{first eigenfunction} $\psi_1$ of the Laplace operator, which has to do with the first mode of a clamped membrane. We know that $\psi_1$ is a solution of the problem
\begin{equation}
\label{first-eigenfunction}
\De \psi_1+\la_1\,\psi_1=0 \ \mbox{ and } \ \psi_1>0 \ \mbox{ in } \ \Om, \quad \psi_1=0 \ \mbox{ on } \ \Ga,
\end{equation}
where $\la_1>0$ is called the \textit{first Dirichlet eigenvalue}. Here, we agree that $\psi_1$ is normalized in $L^2(\Om)$, but it is clear that the location of the maximum points does not depend on how the eigenfunction is normalized. 
\par
An estimate of that location gives information on where an eigenfunction concentrates. Furthermore, it is useful to describe the large time behavior of hot spots in a grounded heat conductor, i.e. the maximum points of the solution of the problem
\begin{equation}
\label{heat-problem}
u_t-\De u=0 \ \mbox{ in } \ \Om\times(0,\infty), \ \  u=0 \ \mbox{ on } \ \Ga\times(0,\infty), \ \ 
u=g \ \mbox{ on } \ \Om\times\{ 0\},
\end{equation}
for some initial (non-negative) distribution of temperature $g$.
In fact, by a spectral formula we know that $e^{\la_1 t} u(x,t)\to \widehat{g}_1\,\psi_1(x)$ as $t\to\infty$, where $\widehat{g}_1$ is the scalar product in $L^2(\Om)$ of $g=u(\cdot,0)$ against $\psi_1$. Thus, under suitable sufficient assumptions, we can claim that the set $\cC_t$ of the hot spots of $u$ must converge to the set $\cC_\infty$ of the maximum points of $\psi_1$, in the sense that $\dist(\cC_t, \cC_\infty)\to 0$ as $t\to\infty$.
\par
To the best of our knowledge, in the literature there are mainly two papers dealing with the problem of locating the maximum points of $\psi_1$ or $u(x,t)$. In one, \cite{GJ}, for a \textit{planar} convex domain the location of the (unique) maximum point $x_\infty$ of $\psi_1$ is estimated by comparing it with that of the maximum point of a solution of a suitably constructed one-dimensional Schr\"odinger equation. The bound is universal. 
\par
In \cite{BMS} instead, two types of results have been obtained. One is a bound in the same spirit of \eqref{dist-torsion}, that holds for convex domains in a general Euclidean space.   The method employed is however peculiar to the case of the first eigenfunction in convex domains and its extension to other equations appears to be difficult. The other estimate, still for convex domains in general dimension, also holds for a quite large class of elliptic and parabolic differential equations. It is based on A. D. Alexandrov's reflection principle and states that the relevant maximum point must fall into the so-called \textit{heart} $\heartsuit(\Om)$ of $\Om$, \textit{independently} of the equation considered. The set $\heartsuit(\Om)$ is defined by purely geometrical means. It has some drawbacks, though. In fact, it is somewhat unstable under small perturbations of $\Om$ and its estimation by means of simple geometrical quantities is not easy (see \cite{BM}).
\par
The method introduced in the present paper is more flexible.
In fact, more or less the same arguments used to prove Theorem \ref{th:maximum-point-torsion} can be adapted to the solutions of problems \eqref{first-eigenfunction} and \eqref{heat-problem}.
For the eigenfunction equation, we have the following result.

\begin{thm}
\label{th:maximum-point-eigenfunction}
Let $\Om$ be a bounded domain with a mean convex boundary $\Ga$ and let $z\in\Om$ be a maximum point of the first Dirichlet eigenfunction $\psi_1$, satisfying problem \eqref{first-eigenfunction}. Then it holds that
\begin{equation}
\label{dist-eigenvalue}
d_\Ga(z)\ge \frac{\pi}{2\,\sqrt{\la_1(\Om)}}.
\end{equation}
In particular, we have that 
\begin{equation}
\label{dist-eigenvalue-inradius}
\frac{d_\Ga(z)}{r_\Om}\ge \frac{\pi}{2\,\sqrt{\la_1(B)}},
\end{equation}
where $B$ is the unit ball.
\end{thm}
\noindent
Inequality \eqref{dist-eigenvalue} is sharper than \eqref{dist-eigenvalue-inradius}. However, the right-hand side of \eqref{dist-eigenvalue-inradius} only depends on $N$.
As is well known, $\sqrt{\la_1(B)}$ is the first zero of the Bessel function of order $N/2-1$. In the case $N=3$ we have $\sqrt{\la_1(B)}=\pi$.
Thus, \eqref{dist-eigenvalue-inradius} is slightly ($7\%$) worse than \eqref{dist-torsion}.
\par
Inequality \eqref{dist-eigenvalue-inradius} may be compared to \cite[Ineq. (1.7)]{BMS}:
$$
\frac{d_\Ga(z)}{r_\Om}\ge\left(\frac{N}{2}\right)^{N-1}\frac{\om_{N-1}}{\om_N \la_1(B)^N} \left[\frac{2\,r_\Om}{\diam(\Om)}\right]^{N^2-1}.
$$
Here, $\om_k$ is the volume of the unit ball in $\RR^k$.
This bound was obtained for bounded convex domains in $\RR^N$. It can be shown (see Remark \ref{rem:comparison-bms}) that the right-hand side of this inequality is always much smaller than that of \eqref{dist-eigenvalue-inradius}. Also, it clearly decays to zero for long and thin domains.
\par
In the case of the heat equation we get instead an evolutive bound from below.

\begin{thm}
\label{th:hotspot}
Let $\Om$ be a bounded domain with a mean convex boundary $\Ga$. Let $g$ be a non-negative function of class $C^1(\ol{\Om})$ such that $g=0$ on $\Ga$. Also, suppose that  
$$
\sup_{\Om} \frac{g}{\phi_1}<\infty,
$$
where $\phi_1$ is the solution of \eqref{first-eigenfunction} whose maximum in $\Om$ is normalized to $1$. 
\par
If, for any fixed $t>0$, $z(t)$ denotes any maximum point in $\Om$ of the solution $u=u(x,t)$ of \eqref{heat-problem}, then it holds that
\begin{equation}
\label{dist-heat}
\frac{d_\Ga(z(t))}{r_\Om}\ge \frac{M(t)}{K}\,e^{\la_1(\Om)\, t}.
\end{equation}
Here,
\begin{equation}
\label{maximum-heat}
M(t)=\max_{x\in\ol{\Om}} u(x,t)
\end{equation}
and
$$
K=\sqrt{\la_1(B)}\,\max\left\{\sup_{\Om} \frac{g}{\phi_1}, \max_{\ol\Om}\sqrt{g^2+\frac{|\na g|^2}{\la_1(\Om)}}\right\}.
$$
\end{thm}
Since, by a spectral formula, 
$$
u(x,t)=\frac{\lan g,\phi_1\ran\,\phi_1(x)}{\,\,\,\,\nr\phi_1\nr_{L^2(\Om)}^2}\,e^{-\la_1(\Om) t} \{1+o(1)\} \ \mbox{ as } \ t\to\infty,
$$ 
the right-hand side of \eqref{dist-heat} does not deteriorates to zero as $t\to\infty$, that is the hot spots stay away from $\Ga$ at all times. Moreover, we can compare \eqref{dist-heat} to \eqref{dist-eigenvalue-inradius} by choosing $g=\phi_1$. In this case $u(x,t)=\phi_1(x)\,e^{-\la_1(\Om) t}$, and hence we can compute that $K\le\la_1(B)$  (see Remark \ref{rem:comparison-temperature-eigenfunction}), so that we obtain the bound:
$$
\frac{d_\Ga(z(t))}{r_\Om}\ge \frac1{\sqrt{\la_1(B)}}.
$$
This is slightly worse than \eqref{dist-eigenvalue-inradius}, but substantially consistent with it.
\par
The proofs of Theorems \ref{th:maximum-point-torsion}, \ref{th:maximum-point-eigenfunction}, and \ref{th:hotspot} are not so difficult. They are all based on now classical bounds for the gradient of the relevant solutions (\cite{PP1, PP2, PPV, CGS}). To make our proofs self-contained, we shall recall and adapt the main arguments used in those references.
\par
To affirm the flexibility of this method, we also show that it provides basic estimates of the location of maximum points of solutions of a variety of equations, that can be quasilinear, isotropic and anisotropic, and semilinear. As an instance of this kind of results, here we consider a generalization of the torsional rigidity function to the case of isotropic quasilinear equations. Here below, $\Phi$ is a \textit{Young's function}, satisfying sufficient smoothness and growth assumptions, and $\Psi$ is its \textit{Young's conjugate} (see Section \ref{sec:nonlinear-torsions} for details). 

\begin{thm}
\label{th:maximum-point-torsion-quasilinear}
Let $\Om$ be a bounded domain with mean convex boundary $\Ga$. Let $z\in\Om$ be any maximum point in $\Om$ of the (weak) solution of 
\begin{equation}
\label{quasilinear-torsion}
-\dv\left\{ \Phi'(|\na u|)\,\frac{\na u}{|\na u|}\right\}=N \ \mbox{ in } \ \Om, \quad u=0 \ \mbox{ on } \ \Ga.
\end{equation}
Then we have that
\begin{equation}
\label{dist-torsion-quasilinear}
d_\Ga(z)\ge \frac1{N}\,\Psi^{-1}(N\,\Psi(r_\Om)).
\end{equation}
\end{thm}
The relevant assumptions on $\Phi$ cover the case of the $p$-Laplace operator, for which we set $\Phi(\si)=\si^p/p$ for $p>1$. Inequality \eqref{dist-torsion-quasilinear} thus reads as
$$
\frac{d_\Ga(z)}{r_\Om}\ge \frac{1}{N^{1/p}},
$$
in accordance with \eqref{dist-torsion} and still independent of geometrical quantities.
In Corollary \ref{cor:bound-elasto-plastic}, we also show that, for a quite general choice of $\Phi$, the right-hand side of the last inequality should be replaced by a quantity also depending on the growth parameters of $\Phi$.
\par
Theorem \ref{th:maximum-point-torsion-quasilinear} can be further generalized to the anisotropic case in which the Euclidean norm of the gradient in \eqref{quasilinear-torsion} is replaced by any norm $H$ on $\RR^N$, satisfying suitable sufficient assumptions. In fact, in Section \ref{sec:Wulff} we shall prove the following result for the case of $H$-mean convex boundaries --- the appropriate analog of mean convex boundaries in this setting.

\begin{thm}
\label{th:dist-estimate-anisotropic}
Let $\Om$ be a bounded domain with $H$-mean convex boundary $\Ga$. Let $z\in\Om$ be any maximum point in $\Om$ of the (weak) solution of
\begin{equation*}
%
-\dv \{\na\Phi_H(\na u)\}=N \ \mbox{ in } \ \Om, \quad u=0 \ \mbox{ on } \ \Ga,
\end{equation*}
where $\Phi_H=\Phi\circ H$.
Then, it holds that
\begin{equation*}
d_\Ga^o(z)\ge \frac1{N}\,\Psi^{-1}(N\,\Psi(r^o_\Om)).
\end{equation*}
\end{thm}
\noindent
Here, $d_\Ga^o$ and $r^o_\Om$ are the appropriate analogs of $d_\Ga$ and $r_\Om$ in the norm $H$ (see Section \ref{sec:Wulff} for details).

We shall begin our account by presenting in  Section \ref{sec:torsion} what we think is the easiest setting: that of the torsional rigidity density of a straight bar or the flow velocity of a fluid in a straight pipe (Theorem \ref{th:maximum-point-torsion}). The simple setting will allow us to dwell on some further details and extensions to more general domains. In  this same section,  we will also present a similar estimate for positive solutions of semilinear equatons (see Theorem \ref{th:maximum-point-small-diffusion}). Section \ref{sec:torsion} ends with the description of the relationship of the maximum points of the torsional rigidity function and those of a related problem in dependence of a diffusion parameter. 
\par
In Section \ref{sec:eigenfunction-heat}, we consider the first Dirichlet eigenfunction for $-\De$ and the case of the heat equation. We prove and compare Theorems \ref{th:maximum-point-eigenfunction} and \ref{th:hotspot}. In these frameworks, the pointwise estimate from below for the relevant solution is not needed. 
\par
Section \ref{sec:nonlinear-torsions} contains a basic introduction to Young's functions, the proof of Theorem
\ref{th:maximum-point-torsion-quasilinear}, and an extension of that theorem to the case of semilinear source terms. 
\par
We conclude our paper with Section \ref{sec:Wulff}, in which we consider quite general anisotropic operators and prove Theorem \ref{th:dist-estimate-anisotropic}.
\par
To avoid unnecessary technicalities, differently from what done in Section \ref{sec:torsion}, in the remaining sections we decided to limit our description to the elegant case of a domain with mean convex boundary $\Ga$ of class $C^{2,\ga}$ for some $\ga\in (0,1]$. The restriction on the regularity of $\Ga$ can be removed by an appropriate approximation argument. We shall present this argument for the case discussed in Section \ref{sec:torsion} (see Lemma \ref{lem:bound-gradient}) and omit it for those considered in Sections \ref{sec:eigenfunction-heat}--\ref{sec:Wulff}, since it is not our purpose to discuss here the optimal regularity assumptions.

\bigskip

 \section{Maximum points of the torsional rigidity function}
\label{sec:torsion}
In this section, we shall present our results on the location of maximum points of the classical torsional rigidity function $u$ defined by \eqref{torsion}. We will consider domains with various geometries. 
As a reference case, we choose that in which $\Om$ is a bounded domain with \textit{mean convex} boundary $\Ga$. Thus, we assume that $\Ga$ is of class $C^2$ and its mean curvature $\cM$ with respect to the interior normal is non-negative. With this choice, convex domains have non-negative principal curvatures, and hence mean convex boundary.

\subsection{Bounds for $\mathbf u$ and its gradient}
\label{subsec:bounds-from-below}
The first step of our argument is a pointwise bound from below  for $u$ in terms of the distance $d_\Ga(x)$ of a point $x\in\Om$ to $\Ga$. This is the content of \cite[Lemma 3.1]{MP2} that, for the reader's convenience, we recall here below.

\begin{lem}
\label{lem:relationdist}
Let $\Om\subset\RR^N$, $N\ge 2$, be a bounded domain. Let $u\in C^0(\ol{\Om})\cap C^2(\Om)$ satisfy the problem \eqref{torsion}.
Then
\begin{equation}
\label{relationdist}
u(x)\ge\frac12\, d_\Ga(x)^2 \ \mbox{ for every } \ x\in\ol{\Om}.
\end{equation}
\end{lem}

\begin{proof}
For a fixed $x\in\Om$, let $r=d_\Ga(x)$ and consider the ball $B=B_r(x)$. 
Let $w^r$ be the solution of \eqref{torsion} in $B$, that is $w^r(y)=(r^2-|y-x|^2)/2$. By comparison we have that $u\ge w^r$ on $\ol{B}$ and hence, in particular, at the center of $B$, that is $u(x) \ge w(x)=r^2/2=d_\Ga(x)^2/2$. 
\end{proof}

\par
Next, we recall an inequality for $|\na u|$ that can be found in \cite{PP1} for dimension $N=2$ (for a proof in a more general setting, we refer to \cite{CGS}). For our aims, in the following lemma we collect, adapt to the case of general dimension,  and re-organize some results contained in \cite{PP1}.

\begin{lem}
\label{lem:bound-gradient}
Let $\Om\subset\RR^N$, $N\ge 2$, be a bounded domain with boundary $\Ga$ of class $C^2$.
Let $u\in C^1(\ol{\Om})\cap C^2(\Om)$ satisfy problem \eqref{torsion}. Set
$$
G=\max_\Ga |\na u| \ \mbox{ and } \ \cM_0^-=\max_\Ga \cM^-,
$$
where $\cM^-=\max(-\cM,0)$.
Then the function defined by
$$
P=\frac12\,|\na u|^2+[N+(N-1)\,\cM_0^-\,G]\,\Bigl[u-\max_{\ol{\Om}} u\Bigr] \ \mbox{ on } \ \ol{\Om}
$$ 
attains its maximum at some critical point of $u$, and hence it holds that
\begin{equation}
\label{gradient-bound-general}
|\na u|^2\le 2\,[N+(N-1)\,\cM_0^-\,G]\,\Bigl[\max_{\ol{\Om}} u-u\Bigr] \ \mbox{ on } \ \ol{\Om}.
\end{equation}
In particular, if $\Ga$ is mean convex, we have that
$$
|\na u|^2\le 2\,N\,\Bigl[\max_{\ol{\Om}} u-u\Bigr] \ \mbox{ on } \ \ol{\Om}.
$$
\end{lem}

\begin{proof}
(i) We first assume that $\Ga$ is of class $C^{2,\ga}$ for some $\ga\in (0,1]$. Then, the standard regularity theory ensures that $u\in C^{2,\ga}(\ol{\Om})\cap C^\infty(\Om)$, and hence that $P\in C^{1,\ga}(\ol{\Om})\cap C^\infty(\Om)$.
\par
Let 
$$
P=\frac12\,|\na u|^2+\be\,[u-\max_{\ol{\Om}} u];
$$
it turns out that 
Next, by straightforward calculations, we obtain the identity:
\begin{multline}
\label{P-identity-1}
|\na u|^2 \De P-|\na P|^2+2\be\,\na u\cdot\na P= \\
|\na u|^2 |\na^2 u|^2-|\na^2 u\, \na u|^2+\be\,(\be-N)\,|\na u|^2 \ \mbox{ in } \ \Om.
\end{multline}
Also,
$$
P_\nu=|\na u|\,\left\{\frac{\lan \na^2 u\,\na u, \na u\ran}{|\na u|^2}+\be\right\}=|\na u|\,\{\be-N+(N-1)\,\cM\,|\na u|\} \ \ \mbox{ on } \ \Ga, 
$$
where we have used the identity
\begin{equation}
\label{reilly-identity}
\De u=u_{\nu\nu}-(N-1)\,\cM\,|\na u| \ \mbox{ on } \ \Ga,
\end{equation}
and the fact that the inward unit normal $\nu$ equals $\na u/|\na u|$ at points on $\Ga$, since $\Ga$ is the boundary of the set where $u$ is positive. Hence, we obtain the inequality
\begin{equation}
\label{P-identity-2}
P_\nu\ge |\na u|\,\{\be-N-(N-1)\,\cM_0^-\,G\} \ \mbox{ on } \ \Ga.
\end{equation}
Thus, if we choose
$$
\be=N+(N-1)\,\cM_0^-\,G\ge N,
$$
then \eqref{P-identity-1} and \eqref{P-identity-2}, and the fact that $|\na^2 u\, \na u|^2\le |\na u|^2 |\na^2 u|^2$ (by Cauchy-Schwarz inequality)
give that
$$
|\na u|^2 \De P-|\na P|^2+2\be\,\na u\cdot\na P\ge 0 \ \mbox{ in } \ \Om \quad\mbox{and}\quad
P_\nu\ge 0 \ \mbox{ on } \ \Ga.
$$
By the strong maximum principle and Hopf boundary lemma, the last two inequalities give that the maximum of $P$ must be attained at a critical point of $u$. Since $P\le 0$ at the critical points of $u$, we conclude that $P\le 0$ on $\ol{\Om}$, and our claim is proved.
\par
(ii) If $\Ga$ is of class $C^2$, we can approximate $\Om$ by a decreasing sequence of domains $\Om_n\supset\ol{\Om}$, with boundaries $\Ga_n$ of class $C^{2,\ga}$ such that the corresponding mean curvatures $\cM_n$ converge to the mean curvature $\cM$ of $\Ga$, uniformly as $n\to\infty$. The corresponding solution  $u_n$ of \eqref{torsion} in $\Om_n$ satisfies \eqref{gradient-bound-general} for every $n\in\NN$, thanks to (i). Since $u_n$ and $\na u_n$ converge uniformly  to $u$ and $\na u$ on $\ol{\Om}$, we conclude that \eqref{gradient-bound-general} holds true for $u$ on $\ol{\Om}$.
\end{proof}

\subsection{The location of maximal torsional points} 

We now proceed to the proof of our first main result. To this aim we set
\begin{equation}
r_\Om=\max_{x\in\ol{\Om}} d_\Ga(x),
\end{equation}
the \textit{inradius} of $\Om$. A point $x_\Om$ attaining the value $r_\Om$ is often called  an \textit{incenter}. A strictly convex domain admits a unique incenter. If the domain is not strictly convex, then it may admit more than one incenter and even a continuum of incenters. For instance, a dumbbell admits two incenters. a rectangle admits a segment of incenters. A (circular) torus admits a circle of incenters (notice that one can construct tori with mean convex boundaries).

\begin{proof}[\bf Proof of Theorem \ref{th:maximum-point-torsion}]
We can apply Lemma \ref{lem:bound-gradient} and obtain that
$$
\frac{|\na u|}{2\,\sqrt{u(z)-u}}\le\sqrt{\frac{N}{2}} \ \mbox{ on } \ \ol{\Om},
$$
since we are assuming that $\cM\ge 0$ on $\Ga$.
We take $x\in\Om$ and proceed as in \cite{PP1}, that is we let $y\in\Ga$ be such that $|x-y|=d_\Ga(x)$  
and, being $u(y)=0$, compute:
\begin{multline*}
\sqrt{u(z)}-\sqrt{u(z)-u(x)}=\int_0^1 \frac{d}{dt} \sqrt{u(z)-u(x+t\,(y-x))}\,dt= \\
\int_0^1 \frac{\na u(x+t\,(y-x))\cdot (x-y)}{2\,\sqrt{u(z)-u(x+t\,(y-x))}}\,dt \le
\sqrt{\frac{N}{2}}\,|x-y|=\sqrt{\frac{N}{2}}\,d_\Ga(x).
\end{multline*}
Thus,  by choosing $x=z$, we have that
\begin{equation}
\label{bound-u(z)}
\sqrt{u(z)}\le\sqrt{\frac{N}{2}}\,d_\Ga(z).
\end{equation}
\par
Finally, we pick an incenter $x_\Om$ of $\Om$ and by Lemma \ref{lem:relationdist} obtain:
$$
\frac1{\sqrt{2}}\,r_\Om=\frac1{\sqrt{2}}\,d_\Ga(x_\Om)\le \sqrt{u(x_\Om)}\le \sqrt{u(z)}\le \sqrt{\frac{N}{2}}\,d_\Ga(z).
$$
Our claim then follows at once.
\end{proof}

\begin{rem}
Notice that \eqref{bound-u(z)} also gives the estimate:
$$
u(z)\le \frac{N}{2}\,r_\Om^2.
$$
This can be found in \cite{PP1}.
\end{rem}

\begin{ex}
The assumption of mean convexity allows domains made of balls connected by \textit{goose-necks} or with \textit{tails} attached.
Theorem \ref{th:maximum-point-torsion} tells us that goose-necks and tails cannot contain a maximum point of $u$, if they are too thin.  
\par
For instance, one can construct a dumbbell-shaped domain in $\RR^3$ with a boundary made by portions of two spheres joined by a portion of a catenoid. The mean curvature of the spheres is constant and positive and that of the catenoid is zero. It is not difficult to smooth out the boundary to obtain a mean convex surface $\Ga$ of class $C^2$. If $\Om$ is the bounded domain having $\Ga$ as a boundary, then Theorem \ref{th:maximum-point-torsion} ensures that
$$
\frac{d_\Ga(z)}{r_\Om}\ge\frac1{\sqrt{3}}=0.57735\cdots.
$$
Thus, if one of the two balls has radius which is smaller than $57\%$ of the other, we have that the maximumn point must fall within (the portion of) the larger sphere, somewhere near its center. 
Notice that a second (local) maximum point (within the smaller ball) and a saddle point (within the catenoid) may be present in $\Om$. 
\end{ex}

\begin{figure}[htbp]
\centering
\includegraphics[scale=.33]{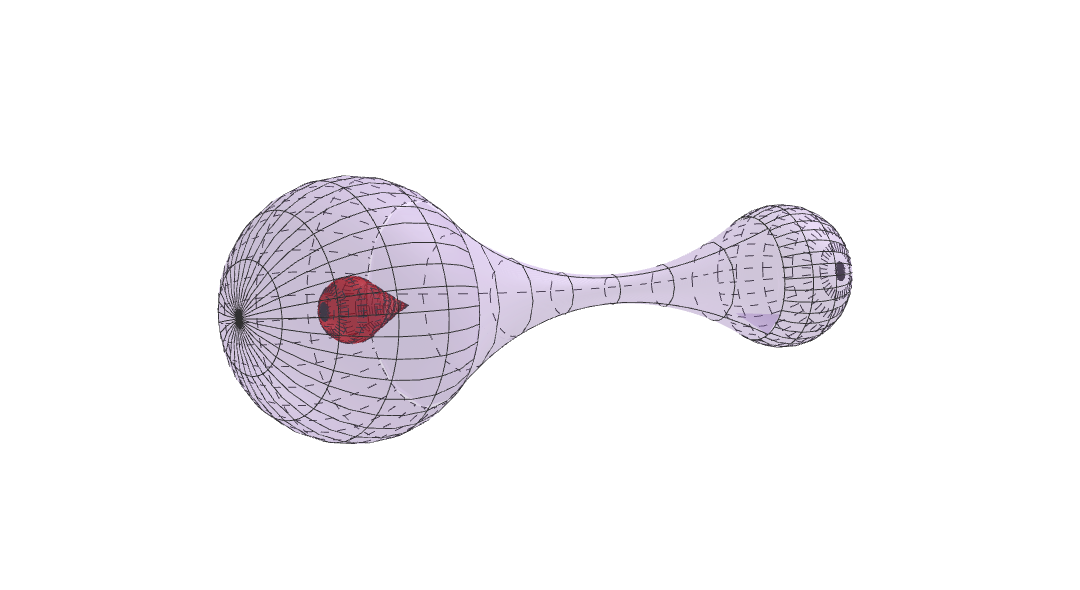}
\caption{The right spherical end of the dumbbell is too small compared to the left one. The maximum point of the torsional rigidity density $u$ must thus fall into the dark domain parallel to the boundary.}
\label{fig:dumb-bell}
\end{figure}

If $\Om$ is convex, it is well-known that $u$ has only one maximum point (see \cite{BL}, \cite{Ko}, \cite{Ma}). 
In our second result, we thus consider this case and obtain an improvement of \eqref{dist-torsion}, based on the \textit{John's ellipsoid} $E_a(c)$ related to $\Om$. This is the ellipsoid of maximum volume contained in $\Om$ (see \cite{Jo}, \cite{Ga}). It is known that, if $\Om$ is convex, $E_a(c)$ is uniquely determined. Here, $c$ denotes the center of $E_a(c)$, the (positive) components of the vector $a=(a_1, \dots, a_N)$ are the semi-axes of $E_a(c)$, and we agree that $a_1\le\cdots\le a_N$.
\begin{thm}
\label{th:ellipse-bound}
Let $\Om$ be a convex domain in $\RR^N$ and $E_a(c)$ be its John's ellipsoid. Let $z\in\Om$ be the maximum point in $\Om$ of the solution $u$ of \eqref{torsion}. Then we have that
\begin{equation*}
\label{dist-torsion-ellipse}
\frac{d_\Ga(z)}{r_\Om}\ge \frac1{\sqrt{N}}\,\max\left[1, \frac{m_{-2}(a)}{r_\Om}\right],
\end{equation*}
where
$$
m_{-2}(a)=\left(\frac1{N}\,\sum\limits_{i=1}^N a_i^{-2}\right)^{-1/2}
$$
is the $(-2)$-mean of the numbers $a_1, \dots, a_N$.
\end{thm}

\begin{proof}
The solution $w$ of \eqref{torsion} in $E_a(c)$ is easily computed as
$$
\displaystyle w(y)=\frac{N}{2}\,\frac{1-\sum\limits_{i=1}^N\left(\frac{y_i-c_i}{a_i}\right)^2}{\sum\limits_{i=1}^N a_i^{-2}} \ \mbox{ for } \ y\in \ol{E_a(c)}.
$$
By proceeding as in the proof of Lemma \ref{lem:relationdist}, we then infer that 
$$
u(z)\ge u(c)\ge w(c)=\frac{N}{2}\,\left\{\sum\limits_{i=1}^N a_i^{-2}\right\}^{-1}.
$$
Since we already know from the proof of Theorem \ref{th:maximum-point-torsion} that
$d_\Ga(z)\ge \sqrt{2\,u(z)/N}$, we then obtain that $d_\Ga(z)\ge m_{-2}(a)/\sqrt{N}$. Our claim then follows by observing that \eqref{dist-torsion} always holds. 
\end{proof}

\subsection{The case of general domains}

If we use Lemma \ref{lem:bound-gradient} in its full power, we can extend Theorem \ref{th:maximum-point-torsion} to the case of general smooth domains, that is by removing the mean convexity assumption.
In this case, the obtained bound obviously depends on the number $\cM^-_0$.

\begin{cor}[Bound for general domains]
\label{cor:domain-curvature}
Let $\Om$ be a bounded domain in $\RR^N$, $N\ge 2$, with boundary of class $C^2$. Assume that $(N-1) \cM_0^- r_\Om<1$.
\par
Let $z\in\Om$ be any maximum point in $\Om$ of the solution $u$ of \eqref{torsion}. Then we have that
\begin{equation*}
\label{dist-torsion-general}
\frac{d_\Ga(z)}{r_\Om}\ge \sqrt{\frac{1-(N-1) \cM_0^- r_\Om}{N}}.
\end{equation*}
\end{cor}

\begin{proof}
By proceeding as in the proof of Theorem \ref{th:maximum-point-torsion}, this time we obtain the inequalities
$$
\frac{|\na u|}{\sqrt{u(z)-u}}\le 2\,\sqrt{\frac{N+(N-1)\,\cM_0^-\,G}{2}}
$$
and
\begin{equation}
\label{bound-u-distance}
\sqrt{u(z)}\le \sqrt{\frac{N+(N-1)\,\cM_0^-\,G}{2}}\,d_\Ga(z)\le r_\Om\, \sqrt{\frac{N+(N-1)\,\cM_0^-\,G}{2}}.
\end{equation}
Thus,
$$
|\na u|\le 2\,\sqrt{u(z)}\,\sqrt{\frac{N+(N-1)\,\cM_0^-\,G}{2}}\le r_\Om\,[N+(N-1)\,\cM_0^-\,G]
$$
Thanks to our assumption on $\cM_0^- r_\Om $, this information then gives the bound:
$$
G\le \frac{N\,r_\Om}{1-(N-1)\,\cM_0^- r_\Om}.
$$
Thus, the claim of the corollary follows from \eqref{relationdist} and  by inserting this bound into the first inequality in \eqref{bound-u-distance}.
\end{proof}

In alternative to the above bound on $G$, we may use the following one:
\begin{equation*}
\label{diam-ext-radius-gradient-estimate}
G\le 
c_N\, \diam(\Om)\left(1+\frac{\diam(\Om)}{r_e}\right),
\end{equation*}
where $\diam(\Om)$ is the diameter of $\Om$, $r_e$ is the radius of the largest ball contained in $\RR^N\setminus\ol{\Om}$ osculating $\Ga$, and 
$c_N=3/2$ for $N=2$ and $c_N=N/2$ for $N\ge 3$. This estimate was proved in \cite{MP1}
and works if $\Ga$ is of class $C^{1,\al}$ for $0<\al\le 1$ and satisfies the uniform exterior sphere condition with radius $r_e$.
\par
Since $\cM_0^-\le 1/r_e$, based on the last inequality for $G$ and the first inequality in \eqref{bound-u-distance}, we easily derive the following result that removes the restriction on $\cM_0^- r_\Om$.

\begin{cor}
\label{cor:domain-diameter-outradius}
Let $\Om\subset\RR^N$, $N\ge 2$, be a bounded domain with boundary $\Ga$ of class $C^{1,\al}$, $0<\al\le 1$, that satisfies the uniform exterior sphere condition with radius $r_e$.
\par
Let $z\in\Om$ be any maximum point in $\Om$ of the solution $u$ of \eqref{torsion}. Then we have that
\begin{equation*}
\label{dist-torsion-general-2}
\frac{d_\Ga(z)}{r_\Om}\ge \left[N+(N-1)\,c_N\, \frac{\diam(\Om)}{r_e}\left(1+\frac{\diam(\Om)}{r_e}\right)\right]^{-\frac12}.
\end{equation*}
\end{cor}

\bigskip

\subsection{Small and large diffusion} 
\label{subsec:small-diffusion}
We conclude this section by considering a problem that is associated to \eqref{torsion}:
\begin{equation*}
\ve\,\De v^\ve=v^\ve \ \mbox{ in } \ \Om, \quad v^\ve=1 \ \mbox{ on } \ \Ga,
\end{equation*}
where $\ve$ is a positive diffusion parameter.
This problem is related to the torsional rigidity $u$, because the function $u^\ve=N \ve \,(1-v^\ve)$ is the solution of
\begin{equation}
\label{small-large-diffusion}
-\De u^\ve+\ve^{-1} u^\ve=N \ \mbox{ in } \ \Om, \quad u^\ve=0 \ \mbox{ on } \ \Ga.
\end{equation}
This means that $v^\ve=1-u^\ve/(N \ve)$ with $u^\ve\to u$ as $\ve\to\infty$. Notice that $u^\ve$ is always positive by the maximum principle.

\begin{rem}
Varadhan's formula (see \cite{Va} and \cite{BM2}) informs us that 
$$
-\sqrt{\ve}\,\log v^\ve(x)\to  d_\Ga(x) \ \mbox{ as } \ \ve\to 0^+.
$$
Since this convergence is known to be uniform on $\ol{\Om}$, we know that the set $\cC_\ve$ of maximum points of $u^\ve$ --- which is the set of minimum points of $v^\ve$ --- tends to the set $\cC_0$ of the maximum points of $\dist(\cdot,\Ga)$, in the sense that $\dist(\cC_\ve, \cC_0)\to 0$ as $\ve\to 0^+$.
In other words, we can infer that for any sequence $\{z^\ve\}_{\ve>0}$ with $z^\ve\in\cC_\ve$, $\ve>0$, it holds that
$$
\lim_{\ve\to 0^+} \frac{\dist(z^\ve, \Ga)}{r_\Om}=1.
$$
In fact, by the uniform convergence of $u^\ve$, any converging subsequence $\{z^\ve\}_{\ve>0}$ converges to a maximum point of $\dist(\cdot, \Ga)$, and hence $\dist(z^\ve, \Ga)\to r_\Om$ as $\ve\to 0^+$.
\end{rem}
\par
The aim of this subsection is now to derive a bound similar to \eqref{dist-torsion} for the maximum points of $u^\ve$ (or the minimum points of $v^\ve$) and to study its evolution in dependence of the diffusion parameter $\ve$ as it goes to $\infty$. 
\par
To proceed further, we need a gradient bound for $u^\ve$, similar to that of  Lemma \ref{lem:bound-gradient}. As a matter of fact, by a little more effort, one can obtain such a bound for any solution of the problem
\begin{equation}
\label{laplace-semilinear}
-\De u=f(u) \ \mbox{ and } \ u\ge 0\ \mbox{ in } \ \Om, \quad u=0 \ \mbox{ on } \ \Ga,
\end{equation}
where $f\in C^1(\RR)$. For later use and the reader's convenience, here below we adjust and prove the statements contained in \cite{PP2} and \cite{CGS}.

\begin{lem}[Gradient estimate for semilinear equations]
\label{lem:gradient-estimate-semilinear}
Let $\Om\subset\RR^N$, $N\ge 2$, be a bounded domain with mean convex boundary $\Ga$.
Let $u$ be a solution of class $C^1(\ol{\Om})\cap C^2(\Om)$ of \eqref{laplace-semilinear} and set
$$
M=\max_{\ol{\Om}} u.
$$
\par
Suppose that $f\in C^1(\RR)$ is such that
$$
\int_u^M f(s)\,ds\ge 0 \ \mbox{ for } \ 0\le u\le M.
$$
Then, the function defined by
$$
\frac12\,|\na u|^2-\int_u^M  f(\si)\,d\si\ \mbox{ on } \ \ol{\Om}
$$ 
attains its maximum at some critical point of $u$, and hence it holds that
\begin{equation}
\label{gradient-semilinear}
|\na u|^2\le 2\,\int^M_u f(\si)\,d\si\ \mbox{ on } \ \ol{\Om}.
\end{equation}
\end{lem}

\begin{proof}
Up to the usual approximation argument, we can only present our proof in case $\Ga$ is of class $C^{2,\ga}$. Set
$$
P=\frac12\,|\na u|^2-\int_u^M f(\si)\,d\si.
$$
We then compute:
$$
\na P=\na^2 u\,\na u+f(u)\,\na u \ \mbox{ and } \ \De P=|\na^2 u|^2-f(u)^2.
$$
In the last identity, we have used the differential equation in \eqref{laplace-semilinear} and its gradient.
We then easily get the identity:
$$
|\na u|^2 \De P-|\na P|^2+2\,f(u)\,\lan\na u,\na P\ran= 
|\na u|^2 |\na^2 u|^2-|\na^2 u\,\na u|^2.
$$
Thus, we have that
\begin{equation}
\label{subsolution-f(u)}
|\na u|^2 \De P-|\na P|^2+2\,f(u)\,\lan\na u,\na P\ran\ge 0 \ \mbox{ in } \ \Om,
\end{equation}
since $|\na^2 u\,\na u|^2\le |\na u|^2 |\na^2 u|^2$. 
Next, we also have that
\begin{multline*}
P_\nu=|\na u|\,u_{\nu\nu}+f(0)\,u_\nu= \\ |\na u|\,\{ (N-1)\,|\na u|\,\cM-f(0)\}+f(0)\,|\na u|= 
(N-1)\,|\na u|^2 \cM \ \mbox{ on } \ \Ga,
\end{multline*}
from \eqref{reilly-identity}, \eqref{laplace-semilinear}, and
since $\nu=\na u/|\na u|$ on $\Ga$. Thus, $P_\nu\ge 0$ on $\Ga$,
being $\Ga$ mean convex.  
As observed before, this inequality and \eqref{subsolution-f(u)} tell us that the maximum of $P$ cannot be attained at a boundary point, by the strong maximum principle and Hopf's boundary lemma.
\par
All in all, the maximum of $P$ must be attained at a critical point of $u$ at which
$$
P= -\int_u^M f(\si)\,d\si\le 0,
$$
and hence $P\le 0$ on $\ol{\Om}$.
\end{proof}

Based on Lemma \ref{lem:gradient-estimate-semilinear}, we obtain the following estimate.  
\begin{thm}
\label{th:maximum-point-small-diffusion}
Let $\Om$ be a bounded domain with mean convex boundary $\Ga$. Let $f\in C^1(\RR)$ 
and set
$$
F(s)=\int_0^s f(\si)\,d\si, \ s\in\RR.
$$
\par
If $z\in\Om$ is any (global) maximum point in $\Om$ of the solution $u$ of \eqref{laplace-semilinear}, then we have that
\begin{equation}
\label{dist-semilinear}
r_\Om\ge d_\Ga(x)\ge \frac1{\sqrt{2}}\int_0^{u(x)}\frac{ds}{\sqrt{F(u(z))-F(s)}} \ \mbox{ for } \ x\in\Om.
\end{equation}
\end{thm}

\begin{proof}
Take $x\in\Om$ and let $y\in\Ga$ be such that $|x-y|=d_\Ga(x)$. Then,
compute:
\begin{multline*}
\frac{d}{d\tau}\int_{u(x+\tau (y-x))}^{u(x)}\frac{ds}{\sqrt{2\,[F(u(z))-F(s)]}}= \\
\frac{\na u(u(x+\tau (y-x)))\cdot(x-y)}{\sqrt{2\,[F(u(z))-F(u(x+\tau (y-x)))]}} \le 
|x-y|=d_\Ga(x),
\end{multline*}
thanks to \eqref{gradient-semilinear}. Integrating in $\tau$ on $[0,1]$ thus gives \eqref{dist-semilinear}, since $u(y)=0$.
\end{proof}

\bigskip

Next, we choose $f(\si)=N-\si/\ve$, that gives $2 F(s)=\ve\, [N^2- (N-s/\ve)^2]$, and analyse the behavior of the points in $\cC_\ve$ as $\ve\to\infty$. 

\begin{cor}
\label{cor:dist-small-large-diffusion}
Let $\Om$ be a bounded domain with mean convex boundary $\Ga$. For $\ve>0$, let $u^\ve$ be the solution of \eqref{small-large-diffusion}.
\par
If $z_\ve\in\cC_\ve$, then it holds that
\begin{equation*}
d_\Ga(z_\ve)\ge  
\sqrt{\ve}\, \cosh^{-1} \left[\frac{N}{N-u^\ve(z_\ve)/\ve}\right],
\end{equation*}
where $\cosh^{-1}:[1,\infty)\to[0,\infty)$ is the inverse function of the hyperbolic cosine $\cosh:[0,\infty)\to[1,\infty)$.
\end{cor}

\begin{proof}
The inequality follows by setting $x=z_\ve$ and $f(\si)=N-\si/\ve$ in \eqref{dist-semilinear}, and by computing the integral.
\end{proof}

\begin{rem}
\par
By proceeding further, we have that
\begin{equation}
\label{bound-small-large-diffusion}
d_\Ga(z_\ve)\ge  
\sqrt{\ve}\, \cosh^{-1} \left[\frac{N}{N-q^\ve(r_\Om)/\ve}\right],
\end{equation}
where
$$
q^\ve(r)=N\,h^\ve(0)\int_0^{r}\left(\int_0^s \si^{N-1} h^\ve(\si)\,d\si\right)\frac{ds}{s^{N-1} h^\ve(s)^2}
$$
and
$$
h^\ve(\si)=\int_0^\pi e^{\frac{\si}{\sqrt{\ve}}\cos\te} (\sin\te)^{N-2} d\te.
$$
\par
In fact, by comparing $u^\ve$ to the solution $w^r$ of 
\eqref{small-large-diffusion}  in the ball $B_r(x)$ with $r= d_\Ga(x)$, we infer that $u^\ve\ge w^r$ on 
$\ol{B_r(x)}$, and hence $u^\ve(x)\ge w^{ d_\Ga(x)}(x)$. Thus, by taking an incenter $x_\Om$, we have that
$w^{r_\Om}(x_\Om)\le u^\ve(x_\Om)\le u^\ve(z_\ve)$. Corollary \ref{cor:dist-small-large-diffusion} then gives \eqref{bound-small-large-diffusion} since $w^{r_\Om}(x_\Om)=q^\ve(r_\Om)$.
\par
It is easily seen that, as $\ve\to\infty$, $q^\ve(r_\Om)\to r_\Om^2/2$, and hence the right-hand side of \eqref{bound-small-large-diffusion} tends to $r_\Om/\sqrt{N}$, in accordance with \eqref{dist-torsion}.
\end{rem}

\bigskip

\section{On the location of hotspots in a grounded heat conductor}
\label{sec:eigenfunction-heat}

In this section, we shall treat the parabolic case and the case of the first eigenfunction, which are intimately connected.

\subsection{The hot spot for large times}
As is well known, the first Dirichlet eigenfunction $\psi_1$ of $-\De$ in $\Om$, that we assume to have unitary norm in $L^2(\Om)$, controls the behaviour of the solution of \eqref{heat-problem} for large times.  We shall denote by $\la_1(\Om)$ the eigenvalue corresponding to $\psi_1$. We know that $\psi_1$ is a solution of the problem:
\begin{equation}
\label{eigenfunction}
\De u+\la\,u=0 \ \mbox{ in } \ \Om, \quad u=0 \ \mbox{ on } \ \Ga,
\end{equation}
for some $\la\in\RR$. If $\la=\la_1(\Om)$,
$\psi_1$ can be assumed to be positive in $\Om$.  The following inequality holds for bounded domains with a mean convex boundary and directly follows from  Lemma \ref{lem:gradient-estimate-semilinear}, by choosing $f(u)=\la_1(\Om)\, u$:
\begin{equation}
\label{max-max-eigenfunction}
|\na \psi_1|^2\le \la_1(\Om)\,(M_1^2-\psi_1^2) \ \mbox{ on } \ \ol{\Om} \ \mbox{ with } \ M_1=\max_{\ol{\Om}} \psi_1.
\end{equation}

\begin{proof}[\bf Proof of Theorem \ref{th:maximum-point-eigenfunction}]
Let $y$ be the nearest point to $z$ in $\Ga$. Then \eqref{dist-eigenvalue} follows from:
\begin{multline*}
\frac{\pi}{2}=\int_0^1 \frac{d}{d\te} \arcsin\left[\frac{\psi_1(y+\te (z-y))}{M_1}\right] d\te= \\
\int_0^1 \frac{\na \psi_1(y+\te (z-y))\cdot (z-y)}{\sqrt{M_1^2-\psi_1(y+\te (z-y))^2}} d\te \le 
 \sqrt{\la_1(\Om)}\, |z-y|=\sqrt{\la_1(\Om)}\, d_\Ga(z).
\end{multline*}
Here, we have used Cauchy-Schwarz inequality and \eqref{max-max-eigenfunction}.
\par
Let $B_{r_\Om}$ be a maximal ball contained in $\Om$. Then, we have that
\begin{equation}
\label{dilation-faber-krahn}
\la_1(\Om)\le\la_1(B_{r_\Om})=\frac{\la_1(B)}{r_\Om^2},
\end{equation}
by the monotonicity and the scaling properties of $\la_1$.
Therefore, \eqref{dist-eigenvalue-inradius} easily follows from \eqref{dist-eigenvalue}.
\end{proof}

\begin{rem}
\label{rem:comparison-bms}
Theorem \ref{th:maximum-point-eigenfunction} greatly improves \cite[Theorem 2.7 and 2.8]{BMS}. In particular, inequality \eqref{dist-eigenvalue-inradius} may be compared to \cite[Ineq. (1.7)]{BMS}:
$$
\frac{ d_\Ga(z)}{r_\Om}\ge\left(\frac{N}{2}\right)^{N-1}\frac{\om_{N-1}}{\om_N \la_1(B)^N} \left[\frac{2\,r_\Om}{\diam(\Om)}\right]^{N^2-1},
$$
that was obtained for bounded convex domains in $\RR^N$, by an argument reminescent of that used to prove Alexandrov-Bakelman-Pucci maximum principle. 
\par
Notice that unlike in \eqref{dist-eigenvalue-inradius} the right-hand side in the last inequality depends on the \textit{eccentricity} $2 r_\Om/\diam(\Om)$ of the convex domain $\Om$, that becomes arbitrarily small for long and thin domains.
\par
Also, we have that
$$
\left(\frac{N}{2}\right)^{N-1}\!\!\!\!\frac{\om_{N-1}}{\om_N \la_1(B)^N} \left[\frac{2\,r_\Om}{\diam(\Om)}\right]^{N^2-1}\le\left(\frac{N}{2}\right)^{N-1}\!\!\!\!\frac{\om_{N-1}}{\om_N \la_1(B)^N} \le
\frac{\pi}{2\sqrt{\la_1(B)}},
$$
thanks to the explicit value of $\la_1(B)$.
\end{rem}

\begin{rem}[The Lane-Emden equation]
As an interesting instance in the semilinear case, we just comment on the \textit{Lane-Emden equation}, widely studied in the literature, for instance, in connection with the large time behavior of the \textit{porous medium equation}. The problem we have in mind occurs in the minimization of the Dirichlet energy functional on the unit sphere of $L^q(\Om)$:
$$
\la_q(\Om)=\inf\left\{\int_\Om |\na v|^2 dx: v\in W^{1,2}_0(\Om) \ \mbox{ and } \ \int_\Om |v|^q dx=1\right\}.
$$
This variational problem has solution for $1<q<2^*$, where $2^*$ is the critical \textit{Sobolev's exponent}, that equals $\infty$ for $N=2$ and $2N/(N-2)$ for $N\ge 3$. The relevant minimizer $u$ is the $L^q(\Om)$-normalized solution of the problem
\begin{equation}
\label{lane-emden}
-\De u=\la_q(\Om)\,|u|^{q-2} u \ \mbox{ in } \ \Om, \quad u=0 \ \mbox{ on } \ \Ga.
\end{equation}
It has been recently proved that, for $1<q<2$, the positive least energy solution  of \eqref{lane-emden} are isolated in the $L^1(\Om)$-topology (see \cite{BDF} for all the details). 
\par
We may use for $u$ Lemma \ref{lem:gradient-estimate-semilinear} and the same arguments used in the proof of Theorem \ref{th:maximum-point-small-diffusion}, and obtain:
\begin{equation}
\label{lane-emden-ineq}
d_\Ga(z)\ge \sqrt{\frac{q}{2\la_q(\Om)}}\,\left(\max_{\ol{\Om}}u\right)^{1-q/2} \int_0^1\frac{d\si}{\sqrt{1-\si^q}}.
\end{equation}
For $q=2$ we recover \eqref{dist-eigenvalue}.
\par
Moreover, similarly to \eqref{dilation-faber-krahn}, we get that $\la_q(\Om)\le r_\Om^{-2+N(1-2/q)} \la_q(B)$. Thus, by the fact that
$$
|\Om|^{1/q}\max_{\ol{\Om}}u\ge \nr u\nr_{L^q(\Om)}=1,
$$
we arrive at the following extended version of \eqref{dist-eigenvalue-inradius}:
$$
\frac{d_\Ga(z)}{r_\Om}\ge \sqrt{\frac{q}{2\, \la_q(B)}}\,\int_0^1\frac{d\si}{\sqrt{1-\si^q}} \left(\frac{r_\Om^N}{|\Om|}\right)^{\frac1{q}-\frac12} \ \mbox{ for } \ 1<q\le 2.
$$
\end{rem}

\bigskip

\subsection{The hot spot at any fixed time}
We now turn to the parabolic case, that concerns the problem 
\eqref{heat-problem}.
As already mentioned, the initial distribution of temperature $g$ is a non-negative function of class $C^1(\ol{\Om})$ and vanishes on $\Ga$. 
It is well-known that a bounded solution $u=u(x,t)$ of class $C^1(\ol{\Om}\times[0,\infty))\times C^2(\Om\times(0,\infty))$ of \eqref{heat-problem} exists and is unique under suitable sufficient conditions on $\Om$ and $g$ (see \cite{Fr}). 
\par
It may be interesting to consider the case in which $g\equiv 1$ (or when $g$ does not vanish on $\Ga$).  We need a little more care in this instance, since the data on $\pa (\Om\times (0,\infty))$ is discontinuous. Nevertheless, it is easy to see that a bounded solution of class $C^0(\ol{\Om}\times(0,\infty))\times C^2(\Om\times(0,\infty))$ exists and is unique.
\par
The strong maximum principle tell us that
$$
0<u<\max_{\ol{\Om}} g \ \mbox{ in } \ \Om\times (0,\infty)
$$
and, once a Hopf boundary lemma is applicable 
(see \cite{MR2839047} for optimal conditions), the maximum
$$
M(t)=\max_{x\in\ol{\Om}} u(x,t)
$$
is attained for every $t>0$ at \textit{internal} points, that are called \textit{hot spots} --- the maximum points of the temperature $u$.
We shall denote by $\cH(t)$ the set of hot spots at time $t>0$, that is
$$
\cH(t)=\{ x\in\Om: u(x,t)=M(t)\}.
$$

Versions of Lemmas \ref{lem:bound-gradient} and \ref{lem:gradient-estimate-semilinear} are obtained in \cite{PP2}, \cite{PPV} for the solution of \eqref{heat-problem}. Here below, we use some of those ideas to obtain ad hoc estimates instrumental to our aims.
In what follows, $\phi_1$ is the first Dirichlet eigenfunction, that we normalize by requiring that
$$
\max_{\ol{\Om}}\phi_1=1.
$$

We first recall the following estimate from \cite[Lemma 1]{PPV}.

\begin{lem}
\label{lem:bound-u-eigenfunction}
Let $\Om$ be a bounded domain in $\RR^N$, $N\ge 2$, 
and suppose that $g$ is a non-negative function of class $C^1(\ol{\Om})$, such that $g\equiv 0$ on  $\Ga$.
\par
Let $u=u(x,t)$ be a bounded solution of class $C^1(\ol{\Om}\times[0,\infty))\times C^2(\Om\times(0,\infty))$ of \eqref{heat-problem}. If
$$
\sup_\Om\left(\frac{g}{\phi_1}\right)<\infty,
$$
then
$$
u(x,t)\le \sup_\Om \left(\frac{g}{\phi_1}\right) \phi_1(x)\,e^{-\la_1(\Om) t} \ \mbox{ for } \ (x,t)\in\ol{\Om}\times (0,\infty).
$$

\end{lem}

\begin{proof}
The function defined by 
$$
\sup_\Om \left(\frac{g}{\phi_1}\right) \phi_1(x)\,e^{-\la_1(\Om) t}, \ (x,t)\in\ol{\Om}\times (0,\infty),
$$
is a solution of the heat equation and is zero on $\Ga\times[0,\infty)$. Moreover, it bounds $g$ from above on $\Om\times\{ 0\}$, by construction. The claim then follows from the maximum principle.
\end{proof}

As already declared in the introduction, in this section we limit our description to the fairly general case of mean convex boundaries, that considerably simplifies matters.

\begin{lem}[A bound for the gradient of $u$]
Let $\Om$ be a bounded domain with a mean convex boundary $\Ga$. Suppose that $u\in C^1(\ol{\Om}\times[0,\infty))\times C^2(\Om\times(0,\infty))$ is the solution of \eqref{heat-problem} with $g\in C^1(\ol{\Om})$ and $g\not\equiv 0$. Then, for $\al\in\RR$,
the function $Q$, defined on $\ol{\Om}\times [0,\infty)$ by
$$
Q(x,t)=\frac12\,e^{2\al t} \bigl\{|\na u(x,t)|^2+\al\,u(x,t)^2\bigr\} \ \mbox{ for } \ (x,t)\in\ol{\Om}\times [0,\infty),
$$
attains its maximum value either at a critical point of $u$ or at a point in $\Om\times\{ 0\}$.
\end{lem}

\begin{proof}
As usual, up to an approximation argument we can assume that $\Ga$ is of class $C^{2,\ga}$, so that the standard regularity theory gives that $u$ has H\"older continuous second derivatives on $\ol{\Om}$ (see \cite{Fr}).
\par
Next, as explained in \cite{PP2}, $Q$ satisfies the differential inequality
$$
|\na u|^2 (\De Q-Q_t)\!-e^{-2\al t} |\na Q|^2\!+2\al\,u\,\na u\cdot\na Q\ge 0 \ \mbox{ in } \ \Om\times(0,\infty).
$$
Indeed, straightforward computations with the help of the first equation in \eqref{heat-problem} give:
\begin{multline*}
|\na u|^2 (\De Q-Q_t)\!-e^{-2\al t} |\na Q|^2\!+2\al\,u\,\na u\cdot\na Q= \\
e^{2\al t}\{|\na^2 u|^2 |\na u|^2\!-\!|\na^2 u\, \na u|^2\}.
\end{multline*}
\par
Thus, since the equation is parabolic away from the critical points of $u$, for any $T>0$ the maximum principle insures that the maximum value of $Q$ on $\ol{\Om}\times[0,T]$ can be attained either on $(\Om\times\{ 0\})\cup(\Ga\times[0,T])$ or at a critical point of $u$.
\par
Next, suppose by contradiction that $(x_0,t_0)$ is a point in $\Ga\times(0,T]$ at which $Q$ attains its maximum value. Then by the Hopf's boundary lemma we must have that either $Q$ is constant on $\ol{\Om}\times [0, t_0]$ or $Q_\nu<0$ at $(x_0,t_0)$. Thus, in the latter case
$$
0>Q_\nu=e^{2\al t_0} |\na u|\,u_{\nu\nu} \ \mbox{ and hence } \ u_{\nu\nu}<0 \ \mbox{ at } \ (x_0,t_0).
$$
On the other hand, the first two equations in \eqref{heat-problem} and the identity \eqref{reilly-identity} give that
\begin{equation}
\label{reilly-identity-2}
0=u_t=\De u=u_{\nu\nu}-(N-1)\,\cM\,|\na u| \ \mbox{ on } \ \Ga.
\end{equation}
Thus, $u_{\nu\nu}\ge 0$ on $\Ga$, since $\Ga$ is mean convex, and hence we have reached a contradiction at $(x_0,t_0)$. Therefore, $Q$ must be constant, say $Q_0$,  on $\ol{\Om}\times [0, t_0]$. Now, since $g\not\equiv 0$ and $Q$ is continuous on $\ol{\Om}\times [0, t_0]$, $Q_0$ must be positive. Thus, in particular we must have that 
$$
e^{2\al t}\,|\na u(x,t)|^2=Q_0>0 \ \mbox{ and } \ 0=Q_\nu=e^{2\al t} |\na u(x,t)|\,u_{\nu\nu}(x,t) 
$$
for $(x,t)\in\Ga\times [0,t_0]$. This information, together with \eqref{reilly-identity-2}, gives that $\cM\equiv 0$ on $\Ga$, and this is a contradiction, being $\Ga$ compact (e.g. $H\equiv 0$ contradicts Minkowski's identity $\int_\Ga H\,\lan x,\nu(x)\ran\,dS_x=|\Ga|$).
\par
All in all, $Q$ cannot attain its maximum value on $\Ga\times(0,T]$ and hence that value can be attained either at a critical point of $u$ or initially. 
\end{proof}

We are now in position to prove our estimate on the location of hot spots.

\begin{proof}[\bf Proof of Theorem \ref{th:hotspot}]
Set $\la=\la_1(\Om)$ and $\phi=\phi_1$ to make notations simpler.  
By choosing $\al=\la$ in Lemma 4.2, we have that either
$$
\left[ | \na u |^2 + \la\, u^2 \right] e^{2 \la t} \le \max_\Om \left[ |\na g|^2 + \la\, g^2 \right]
$$
or
$$
\left[ | \na u |^2 + \la\, u^2 \right] e^{2 \la t} \le \la\, u (\xi,\tau )^2 e^{2 \la \tau }
$$
for some critical point $(\xi,\tau)$ of $u$ in $\Om\times(0,\infty)$.
Now, Lemma \ref{lem:bound-u-eigenfunction} gives that
$$
u ( \xi,\tau) \le \phi(\xi)\, \sup_\Om \left(\frac{g}{\phi}\right) e^{-\la\tau} \le \sup_\Om\left(\frac{g}{\phi}\right) e^{-\la \tau},
$$
being $\phi$ normalized.
Hence, we infer that
\begin{equation*}
|\na u|^2 + \la\, u^2 \le K_\Om^2\, e^{- 2 \la t} \ \mbox{ in } \ \Om\times(0,\infty),
\end{equation*}
where
$$
K_\Om=\sqrt{\la_1(\Om)}\,\max\left\{\sup_{\Om} \frac{g}{\phi_1}, \max_{\ol\Om}\sqrt{g^2+\frac{|\na g|^2}{\la_1(\Om)}}\right\}.
$$
that yields:
\begin{equation*}
\frac{|\na u|}{\sqrt{ (K_\Om\, e^{ -\la t})^2 - (\sqrt{\la} u)^2 }} \le  1 \ \mbox{ in } \ \Om\times(0,\infty).
\end{equation*}

Next, as usual, take $z(t) \in \cH (t)$ and let $y(t)$ be the nearest point to $z(t)$ in $\Ga$.
Since
$$
\frac{ \sqrt{\la} \, M(t) \, e^{\la t} }{K_\Om} \le \arcsin \left( \frac{ \sqrt{\la} \, M(t) \, e^{\la t} }{K_\Om} \right) = \arcsin \left( \frac{ \sqrt{\la} \, u( z(t) , t) \, e^{\la t} }{K_\Om} \right),
$$
by setting $\xi(t)=y(t)+\te\,[z(t) -y(t)]$, we then have that 
\begin{multline*}
\frac{ \sqrt{\la} \, M(t) \, e^{\la t} }{K_\Om}\le  \int^1_0 \frac{d}{d \te} \arcsin \left( \frac{ \sqrt{\la} \, u(\xi(t), t) \, e^{\la t} }{K_\Om} \right)  d \te =
\\
\sqrt{ \la } \, \int^1_0  \frac{ \na u(\xi(t), t)\cdot [z(t) - y(t)]}{ \sqrt{ (K_\Om\, e^{-\la t})^2 -[\sqrt{\la} u(\xi(t),t)]^2}}\, d \te
\le
\\
\sqrt{\la}\, |z(t) - y(t)| = \sqrt{\la}\,d_\Ga (z(t)).
$$
\end{multline*}
Thus, \eqref{dist-heat} follows, by observing that $K_\Om\le K \sqrt{\la_1(B)}/r_\Om$, thanks to \eqref{dilation-faber-krahn}.
\end{proof}

\medskip

\begin{rem} 
\label{rem:comparison-temperature-eigenfunction}
The \textit{spectral formula} informs us that
$$
u(x,t)=\sum_{n\in\NN} \widehat{g}_n\,\psi_n(x)\,e^{-\la_n(\Om) t} \ \mbox{ in } \ L^2(\Om),
$$
where $\{\psi_n\}_{n\in\NN}$ is a basis in $L^2(\Om)$ of eigenfunctions of $-\De$ in $\Om$. Since
$\psi_1$ is proportional to $\phi_1$ and $\psi_1$ has unit norm in $L^2(\Om)$, we have that
$$
u(x,t)=\frac{\phi_1(x)\,e^{-\la_1(\Om) t}}{\nr\phi_1\nr_2^2 }\int_\Om g\,\phi_1 dy +\sum_{n=2}^\infty \widehat{g}_n\,\psi_n(x)\,e^{-\la_n(\Om) t},
$$
and hence
$$
u(x,t)=\frac{\phi_1(x)\,e^{-\la_1(\Om) t}}{\nr\phi_1\nr_2^2 }\int_\Om g\,\phi_1 dy +O(e^{-\la_2(\Om) t}) \ \mbox{ as } \ t\to\infty.
$$
It is thus interesting to compare \eqref{dist-heat} to \eqref{dist-eigenvalue} or \eqref{dist-eigenvalue-inradius}.
\par
We can do that by choosing $g=\phi_1$, that satisfies the assumptions of Theorem \ref{th:hotspot}. In this case $u(x,t)=\phi_1(x)\,e^{-\la_1(\Om)\,t}$ solves the problem \eqref{heat-problem} and the hot spots are the maximum points of $\phi_1$. We have that \eqref{dist-heat} yields that, for any maximum point $z$ of $\phi_1$,
$ d_\Ga(z)/r_\Om\ge K^{-1}$.
We also have that
$$
K\le\sqrt{\la_1(B)}\,\max\left\{1, \max_{\ol\Om}\sqrt{\phi_1^2+\frac{|\na \phi_1|^2}{\la_1(\Om)}}\right\}=\sqrt{\la_1(B)};
$$
in the last inequality we have used \eqref{max-max-eigenfunction}. Thus, we obtain the bound 
$$
\frac{ d_\Ga(z)}{r_\Om}\ge \frac1{\sqrt{\la_1(B)}}.
$$
This bound is poorer than \eqref{dist-eigenvalue-inradius}. However \eqref{dist-heat} appears to be consistent for large times.
 \end{rem}

\begin{rem}
In general, we have that
$$
M(t)\,e^{\la_1(\Om) t}\to \max_{\ol{\Om}} g \ \mbox{ as } \ t\to 0^+ 
$$
and
$$
M(t)\,e^{\la_1(\Om) t}\to \nr\phi_1\nr_2^{-2} \int_\Om g\,\phi_1 dy \ \mbox{ as } \ t\to\infty,
$$
by the spectral formula. The bound in \eqref{dist-heat} can then be computed in the limit cases, accordingly.
\end{rem}

\begin{rem}
Notice that, if $g\equiv 1$, similarly to what derived in Subsection \ref{subsec:small-diffusion}, we have that
$$
\lim_{t\to 0^+} d_\Ga(z(t))=r_\Om,
$$
for any hot spot $z(t)$. This follows from the Varadhan's formula (see \cite{Va} or \cite{BM1}):
$$
\lim_{t\to 0^+} 4t\,\log [1-u(x,t)]=- d_\Ga(x)^2, \ x\in\ol{\Om},
$$
where the convergence is uniform on $\ol{\Om}$.
\end{rem}

\medskip

\section{Quasilinear and semilinear isotropic operators}
\label{sec:nonlinear-torsions}

In this section, we shall extend our results to nonlinear settings. We will consider in more detail the situation of the \textit{torsional rigidity function} for isotropic quasilinear elliptic operators. Then we will turn to the case in which the source term is semilinear, i.e. it only depends on the function $u$. This instance also takes care of various examples of eigenfunctions for nonlinear operators. 
\par
All proofs rely on ad hoc gradient bounds. These are already present in the literature, and hence we will just recall their statements adapted to our aims and notations, rather then offering their often elaborate proofs. 

\subsection{Quasilinear isotropic setting.}
\label{subsec:quasilinear-setting}
We will work in a variational framework, that proves to be quite convenient in this case. Thus, the solutions we will consider will generally be critical points of variational integrals of type
\begin{equation}
\label{variational-Phi}
\int_\Om [\Phi(|\na v|)-F(v)]\,dy,
\end{equation}
among all the functions $v\in W^{1,p}_0(\Om)$ with $p>1$. 
\par
Our assumptions on $\Phi$ and $F$ are sufficient to fit those considered by Caffarelli, Garofalo and Segala in \cite{CGS}.  Thus, $F:\RR\to\RR$ is a non-negative primitive of $f\in C^1(\RR)$.
\par
Also, for the sake of brevity, we will only deal with \cite[Assumption (A)]{CGS} for $\Phi$. Therefore, here, $\Phi:[0,\infty)\to\RR$ is a function of class $C^1([0,\infty))\cap C^2((0,\infty))$ such that $\Phi(0)=\Phi'(0)=0$ and, if we denote by $\na_\xi\Phi$ and $\na_\xi^2\Phi$,  the gradient and Hessian matrix of the function $\RR^N\setminus\{ 0\}\ni\xi\mapsto\Phi(|\xi|)$, and by $e$ and $\cE$ the smallest and largest eigenvalues of $\na_\xi^2\Phi$, it holds that
\begin{equation}
\label{assumptions-Phi}
\begin{array}{ll}
&c\,(a+|\xi|)^{p-1}\le |\na_\xi\Phi(|\xi|)|\le C\,(a+|\xi|)^{p-1}, \\
&c\,(a+|\xi|)^{p-2}\le e(|\xi|)\le\cE(|\xi|)\le C\,(a+|\xi|)^{p-2},
\end{array}
\end{equation}
for every $\xi\ne 0$ and some constants $p>1$, $a\ge 0$, and $0<c\le C$. 
\par
For notational convenience, we set $\phi=\Phi'$, so that $\phi(0)=0$ and
$$
\Phi(\si)=\int_0^\si \phi(s)\,ds, \ \si\in [0,\infty).
$$
It is clear that $\Phi$ is strictly convex and $\phi$ is strictly increasing. Under these assumptions a relevant critical point of the functional is thus a weak solution of the problem
\begin{equation}
\label{quasilinear-semilinear-isotropic}
-\dv\left\{ \phi(|\na u|)\,\frac{\na u}{|\na u|}\right\}=f(u) \ \mbox{ in } \ \Om, \quad u=0 \ \mbox{ on } \ \Ga.
\end{equation}
Due to \cite[Theorem 1]{To},  and our assumptions on $\Phi$, if  $u \in W^{1,p}(\Om) \cap L^\infty(\Om)$ is a weak solution of \eqref{quasilinear-semilinear-isotropic}, then we have that $u \in C^{1,\ga} (\Om)$, and this regularity can be brought up to the (sufficiently smooth) boundary thanks to \cite{Li}.
\par
From convex analysis we know that, being $\Phi(\si)\ge c_1\,\si^p/p$ with $p>1$, the \textit{Young conjugate} function $\Psi$ associated to $\Phi$ is well defined by
\begin{equation}
\label{Young-conjugate}
\Psi(\tau)=\max_{\si\ge 0} [\tau\,\si-\Phi(\si)] \ \mbox{ for } \ \tau\ge 0.
\end{equation}
Thus, by definition, the \textit{Young's inequality} holds true:
$$
\Phi(\si)+\Psi(\tau)\ge \si\,\tau \ \mbox{ for any } \ \si, \tau\ge 0.
$$
If we set $\psi=\Psi'$, it turns out that $\psi$ is the inverse function of $\phi$, that is $\phi(\psi(\tau))=\tau$ and $\psi(\phi(\si))=\si$ for every $\si, \tau\ge 0$.
\par
An important case study occurs when 
\begin{equation}
\label{p-laplace-case}
\Phi(\si)=\frac{1}{p}\,\si^p  \  \mbox{ for } \ \si\in[0,\infty), \ p>1.
\end{equation}
In this instance, the Young's conjugate of $\Phi$ is simply given by $\Psi(\tau)=\tau^{p'}/p'$, where $p'$ is the conjugate exponent of $p$, that is $1/p+1/p'=1$. With this choice of $\Phi$, problem \eqref{quasilinear-semilinear-isotropic} is nothing else than 
the $p$-laplacian case:
\begin{equation}
\label{p-torsion}
-\De_p u=f(u) \ \mbox{in} \ \Om, \quad u=0 \ \mbox{ on } \ \Ga;
\end{equation}
here we (formally) denote:
$$
\De_p u =\dv\{|\na u|^{p-2} \na u\}.
$$
\par
We shall refer to this case as that of the \textit{torsional rigidity} of a long straight bar with cross section $\Om$ for an (isotropic) \textit{elasto-plastic} material.
In fact, in the physical model, the relevant material changes its properties in dependence of the parameter $p$. For values of $p$ near $2$, the material has an \textit{elastic} behavior, whereas when $p$ increases, the material gradually acquires \textit{plastic} properties. Thus, in this sense, we are working in an elasto-plastic setting.

\medskip

\subsection{Torsional rigidity in the elasto-plastic setting}
This subsection is dedicated to present the proof of Theorem \ref{th:maximum-point-torsion-quasilinear} and to detail some of its consequences. In other words, we will consider the solution of \eqref{quasilinear-torsion}, that for convenience we rewrite here with the new adopted notation:
\begin{equation}
\label{quasilinear-torsion-2}
-\dv\left\{ \phi(|\na u|)\,\frac{\na u}{|\na u|}\right\}=N \ \mbox{ in } \ \Om, \quad u=0 \ \mbox{ on } \ \Ga.
\end{equation}

In the set up described in Section \ref{subsec:quasilinear-setting}, the (weak) solution of \eqref{quasilinear-torsion-2} when $\Om$ is a ball $B_r(x)$ is easily computed as
$$
w^r(y)=\Psi(r)-\Psi(|y-x|) \ \mbox{ for } \ y\in B_r(x).
$$
We can then derive the companion to Lemma \ref{lem:relationdist}. 
\begin{lem}[A bound from below]
\label{lem:relationdist-quasilinear}
Let $\Om\subset\RR^N$, $N\ge 2$, be a bounded domain with boundary $\Ga$. Let $u\in C^0(\ol{\Om})\cap C^{1,\ga}(\Om)$, $0<\ga\le 1$, be the (weak) solution of \eqref{quasilinear-torsion-2}. Then
\begin{equation}
\label{relationdist-2}
u(x)\ge \Psi( d_\Ga(x)) \ \mbox{ for every } x\in\ol{\Om}.
\end{equation}
\end{lem}

\begin{proof}
We proceed as usual: for $x\in\Om$ we let $r= d_\Ga(x)$ and consider the ball $B=B_r(x)$. Thus, we obtain the comparison $u\ge w^r$ on $\ol{B_r(x)}$, and hence at $x$ we get:
$$
u(x)\ge w^r(x)=\Psi(r)-\Psi(0)=\Psi( d_\Ga(x)),
$$
as claimed.
\end{proof}

We now need a counterpart of Lemma \ref{lem:bound-gradient}. To avoid further difficulties, we shall limit our discussion to the case in which $\Ga$ is a mean convex surface. We will thus adapt \cite{PP1} to our framework. The statement of Lemma \ref{lem:bound-gradient-quasilinear} here below is slightly different from those contained in \cite{PP1} or \cite{CGS}, because we chose to present the relevant $P$-function in terms of the Young conjugate. We think that the ensuing estimate becomes more instructive.

\begin{lem}[Gradient estimate]
\label{lem:bound-gradient-quasilinear-torsion}
Let $\Om\subset\RR^N$, $N\ge 2$, be a bounded domain with mean convex boundary $\Ga$.
Let $u \in W^{1,p}(\Om) \cap L^\infty(\Om)$ be the weak solution of \eqref{quasilinear-torsion-2}.
Then the function defined by
$$
P=\Psi(\phi(|\na u|))+N\,\Bigl[u-\max_{\ol{\Om}} u\Bigr] \ \mbox{ on } \ \ol{\Om}
$$ 
attains its maximum at some critical point of $u$, and hence it holds that
$$
\Psi(\phi(|\na u|))\le N\,\Bigl[\max_{\ol{\Om}} u-u\Bigr] \ \mbox{ on } \ \ol{\Om}.
$$
\end{lem}

\begin{proof}
As already observed, by \cite{To} and \cite{Li}, we know that $u \in C^{1,\ga}(\ol{\Om})$ for some $\ga\in(0,1]$.
Moreover, since we assume that $\Ga$ is sufficiently regular (up to an approximation argument), the strong comparison principle  (see \cite{CT}) together with a standard barrier argument ensure that $|\na u|$ is strictly positive on $\Ga$. Thus, $u$ gains sufficient extra regularity in a neighborhood of $\Ga$, since it solves a uniformly elliptic equation.

First set
$$
P=\Psi(\phi(|\na u|))+\be\,\Bigl[u-\max_{\ol{\Om}} u\Bigr] ,
$$
where $\be\in\RR$ is to be determined. Notice that 
$$
\frac{d}{d\si}\Psi(\phi(\si))=\si\,\phi'(\si).
$$
Once this remark is done, the proof is obtained by an adaptation of the calculations in \cite{PP1} to our framework. Indeed, we have that
\begin{multline*}
\tr[A(|\na u|)\,\na^2 P]-a(|\na u|)\,|\na P|^2+\\
b(|\na u|)\,(\na u\cdot\na P)^2+c(|\na u|)\,\na u\cdot\na P\ge N^2 \be\,(\be/N-1)\,\frac{|\na u|\,\ep(|\na u|)}{\phi(|\na u|)}
\end{multline*}
and
$$
P_\nu=|\na u|\,\left\{\phi'(|\na u|)\,\frac{\lan \na^2 u\,\na u, \na u\ran}{|\na u|^2}+\be \right\}  \ \mbox{ on } \ \Ga.
$$
In the first inequality we have set:
\begin{equation}
\label{ellipticity-Phi}
\epsilon(\si)=\frac{e(\si)}{\cE(\si)}=\frac{\si\,\phi'(\si)}{\phi(\si)} \ \mbox{ for } \ \si\in[0,\infty)
\end{equation}
and
\begin{eqnarray*}
&\displaystyle  A(\xi)=I+\frac{\ep(|\xi|)-1}{|\xi|^2}\,\xi\otimes\xi, \quad &c(\si)=N\,\frac{(2\be/N-1)\,\ep(\si)+1}{\si\,\phi(\si)}, \\
&\displaystyle a(\si)=\frac{\si\,\ep'(\si)+\ep(\si)}{\si\,\phi(\si)\,\ep(\si)^2}, \quad   &b(\si)=\frac{\si\,\ep'(\si)+\ep(\si)\,[1-\ep(\si)^2]}{2\,\si^3 \phi(\si)\,\ep(\si)^2}.
\end{eqnarray*}

Next, since 
$$
u_{\nu\nu}=\frac{\lan \na^2 u\,\na u, \na u\ran}{|\na u|^2} \ \mbox{ on } \ \Ga,
$$
by the identity \eqref{reilly-identity} and \eqref{quasilinear-torsion-2}, we obtain that
$$
\phi'(|\na u|)\,\frac{\lan \na^2 u\,\na u, \na u\ran}{|\na u|^2}=-N+(N-1)\,\phi(|\na u|)\,\cM \ \mbox{ on } \ \Ga,
$$
so that
$$
P_\nu=|\na u|\,\left\{\be-N+(N-1)\,\phi(|\na u|)\,\cM\right\} \ \mbox{ on } \ \Ga.
$$
\par
Now, since $\cM\ge 0$, if we choose $\be\ge N$
we have that
\begin{multline*}
\tr[A(|\na u|)\,\na^2 P]-a(|\na u|)\,|\na P|^2+
b(|\na u|)\,(\na u\cdot\na P)^2+ 
c(|\na u|)\,\na u\cdot\na P\ge 0,
\end{multline*}
away from the critical points of $u$ in $\Om$, and $P_\nu\ge 0$ on $\Ga$.
By the strong maximum principle and Hopf boundary lemma, the last two inequalities give that the maximum of $P$ must be attained at a critical point of $u$, at which $P\le 0$. Thus, $P\le 0$ on $\ol{\Om}$ for any $\be\ge N$, and hence our claim follows by choosing $\be=N$. 
\end{proof}

\medskip

\begin{proof}[\bf Proof of Theorem \ref{th:maximum-point-torsion-quasilinear}]
We can apply Lemma \ref{lem:bound-gradient-quasilinear-torsion} and, since $M=u(z)$, obtain that
$$
|\na u|\le \psi(\Psi^{-1}(N\,[u(z)-u])) \ \mbox{ on } \ \ol{\Om}.
$$
Set
\begin{equation}
\label{def-chi}
\chi(\si)=\int_0^\si \frac{ds}{\psi(\Psi^{-1}(N\,s))}.
\end{equation}
We take as before $x\in\Om$ and $y\in\Ga$ such that $|x-y|=\dist(x,\Ga)$  
and compute:
\begin{multline*}
\chi(u(z))-\chi(u(z)-u(x))=\int_0^1 \frac{d}{dt} \chi(u(z)-u(x+t\,(y-x))\,dt= \\
\int_0^1 \frac{\na [u(z)- u(x+t\,(y-x))]\cdot (y-x)}{\psi(\Psi^{-1} (N\,[u(z)- u(x+t\,(y-x))]))}\,dt \le
|x-y|= d_\Ga(x).
\end{multline*}
Thus, choosing $x=z$ gives that 
$$
\chi(u(z)) \le  d_\Ga(z).
$$
\par
Now notice that, by the change of variables $N s=\Psi(t)$, we have that
$$
\chi(\si)=\frac1{N}\,\Psi^{-1}(N\,\si).
$$
Therefore, we pick an incenter of $\Om$ and by Lemma \ref{lem:relationdist-quasilinear} obtain that
$\Psi(r_\Om)\le u(x_\Om)\le u(z)$, and hence
$$
\frac1{N}\,\Psi^{-1}(N\,\Psi(r_\Om))\le \Phi(u(z))\le d_\Ga(z).
$$
Our claim is proved.
\end{proof}

When the number $a$ in \eqref{assumptions-Phi} is zero, the right-hand side of \eqref{dist-torsion-quasilinear} can be bounded from below by a quantity that only depends on $N$ and the constants $c$ and $C$. In the following corollary, we will carry out this case. For completeness, in Remark \ref{rem:general-elasto-plastic-bound} below, we shall briefly sketch how to obtain a similar estimate, which however is not independent on $r_\Om$, when $a>0$.

\begin{cor}
\label{cor:bound-elasto-plastic}
Set $1<p<\infty$. Under the assumptions of Theorem \ref{th:maximum-point-torsion-quasilinear}, if \eqref{assumptions-Phi}  holds with $a=0$, then we have that
$$
\frac{d_\Ga(z)}{r_\Om}\ge \left(\frac{c}{N C}\right)^{1/p}.
$$
In particular, in the case of the $p$-laplacian, it holds that
$$
\frac{d_\Ga(z)}{r_\Om}\ge \frac1{N^{1/p}}.
$$
\end{cor}

\begin{proof}
For notational convenience, we set $r=r_\Om$ and $d=d_\Ga(z)$. From \eqref{Young-conjugate} and \eqref{assumptions-Phi}, we have that
$$
\frac{C^{1-p'}}{p'}\,\tau^{p'}\le \Psi(\tau)\le \frac{c^{1-p'}}{p'}\,\tau^{p'} \ \mbox{ for } \ \tau\ge 0.
$$
Thus, \eqref{dist-torsion-quasilinear} gives that
$$
\frac{N C^{1-p'}}{p'}\,r^{p'}\le N\,\Psi(r)\le\Psi(N\,d)\le 
\frac{c^{1-p'} N^{p'}}{p'}\,d^{p'},
$$
that yields our claim. In the case of the $p$-laplacian we have that $c=C$.
 \end{proof}

\begin{rem}
\label{rem:general-elasto-plastic-bound}
When $a>0$, similar calculations give the inequality:
\begin{multline*}
\bigl[N C^{1-p'} r^{p'}-N p' a\,r+N (p'-1)\,C\, a^{p'}\bigr]^+\le N\,\Psi(r)\le\Psi(N\,d)\le \\
\bigl[N^{p'} c^{1-p'} d^{p'}-N p' a\,d+N (p'-1)\,c\, a^{p'}\bigr]^+.
\end{multline*}
 Thus, we can conclude that 
$$
\frac{d_\Ga(z)}{r_\Om}\ge \mu_{N,p}(r_\Om, a, c, C),
$$
for some $\mu_{N,p}(r_\Om, a, c, C)\in [0,1)$.
\end{rem}

\begin{rem}
If $z_p$ is a maximum point for the solution of \eqref{p-torsion} with $f\equiv N$, then, modulo a subsequence, we have that there exists a point $z$ such that
$$
r_\Om\ge d_\Ga(z)=\lim_{p\to\infty}\dist(z_p,\Ga)\ge\lim_{p\to\infty}\frac{1}{N^{1/p}}\,r_\Om=r_\Om.
$$
Thus, $z_p$ converges to an incenter $x_\Om$ as $p\to\infty$.
\end{rem}

\begin{rem}
It is not difficult to obtain an analog of Corollary \ref{cor:domain-curvature} in the elasto-plastic setting by analysing the proofs of Lemma \ref{lem:bound-gradient-quasilinear-torsion} and Theorem \ref{th:maximum-point-torsion-quasilinear}.
\end{rem}

\medskip

\subsection{The quasilinear-semilinear isotropic case}
One can obtain a gradient estimate of the type of Lemma \ref{lem:gradient-estimate-semilinear} for the general quasilinear operators considered in this section. The useful reference is now \cite[Theorem 1.6]{CGS}. In other words, we can consider a positive solution (if any) of the problem 
\begin{equation}
\label{quasilinear-semilinear}
-\dv\left\{ \phi(|\na u|)\,\frac{\na u}{|\na u|}\right\}=f(u) \ \mbox{ in } \ \Om, \quad u=0 \ \mbox{ on } \ \Ga,
\end{equation}
where $f$ is a non-linearity of class $C^1(\RR)$. The following lemma adapts \cite[Theorem 1.6]{CGS} to our aims and notations.

 \begin{lem}[Gradient estimate]
\label{lem:bound-gradient-quasilinear}
Let $\Om\subset\RR^N$, $N\ge 2$, be a bounded domain with mean convex boundary $\Ga$.
Let $u$ be a weak solution of \eqref{quasilinear-semilinear} and suppose that $u$ is of class $C^{1,\ga}(\ol{\Om})$ for some $\ga\in(0,1]$ and $C^2$ near $\Ga$.
Then the function defined by
$$
P=\Psi(\phi(|\na u|))-\int_u^M f(\si)\,d\si \ \mbox{ on } \ \ol{\Om}
$$ 
attains its maximum at some critical point of $u$, and hence it holds that
$$
\Psi(\phi(|\na u|))\le\int_u^M f(\si)\,d\si \ \mbox{ on } \ \ol{\Om}.
$$
\end{lem}

\begin{rem}
\label{rem:quasilinear-semilinear-ineq}
Based on this lemma and thanks to the now usual arguments (see Theorem \ref{th:maximum-point-small-diffusion}), for any maximum point $z\in\Om$ we obtain the inequalities:
\begin{equation}
\label{quasilinear-semilinear-gradient-ineq}
r_\Om\ge  d_\Ga(z)\ge\int_0^{u(z)} \frac{d\si}{\zeta\left(\int^{u(z)}_{u(z)-\si} f(s)\,ds\right)} \ \mbox{ with } \ \zeta= \psi\circ\Psi^{-1}.
\end{equation}
\par
If the operator in \eqref{quasilinear-semilinear} satisfies a comparison principle and the function of $u(z)$ at the right-hand side of \eqref{quasilinear-semilinear-gradient-ineq} is monotone increasing, we may obtain a bound from below for 
$ d_\Ga(z)$ in terms of $r_\Om$. To this aim, for any fixed $x\in\Om$ we must compare $u$ to the radially symmetric positive solution (if any) $w^r(y-x)$ of \eqref{quasilinear-semilinear} in the ball $B_{r}(x)$, with $r= d_\Ga(x)$. For the records, $w(\tau)=w^r(\tau)$ must satisfy the ODE problem:
$$
-[\tau^{N-1} \phi(w')]'=\tau^{N-1} f(w) \ \mbox{ in } \ (0,r), \quad w(r)=0, \ w'(0)=0.
$$
\par
Therefore, after simple manipulations eventually we get:
$$
 d_\Ga(z)\ge \int_0^1 \frac{w^{r_\Om}(0)\,d\si}{\zeta\Bigl(w^{r_\Om}(0)\,\int_{1-\si}^1 f(w^{r_\Om}(0)\,s)\,ds\Bigr)}.
$$
 \end{rem}

\begin{rem}
An interesting case in which a comparison principle does not hold occurs if we choose
$$
\Phi(\si)=\frac{\si^p}{p} \ \mbox{ and } \ f(s)=\la |s|^{p-2} s.
$$
This choice corresponds to the problem:
$$
 -\De_p u=\la\, |u|^{p-2} u \ \mbox{ and } \ u>0 \ \mbox{ in } \ \Om, \quad u=0 \ \mbox{ on } \ \Ga.
$$
The eigenvalue $\la=\la_{1,p}(\Om)$ is the sharp constant in the \textit{Sobolev-Poincar\'e inequality}: 
$$
\la_{1,p}(\Om)\,\int_\Om |u|^p dx\le \int_\Om |\na u|^p dx \ \mbox{ for every } \ u\in W^{1,p}_0(\Om).
$$
\par
Thanks to the formula \eqref{quasilinear-semilinear-gradient-ineq}, we thus get:
$$
 d_\Ga(z)\ge \frac{2}{p} \left[\frac{p-1}{\la_{1,p}(\Om)}\right]^{1/p} \int_0^{\pi/2} (\tan\te)^{2/p-1}  d\te=\frac{1}{p} \left[\frac{p-1}{\la_{1,p}(\Om)}\right]^{1/p} \be(1/p,1/p'),
$$
where $\be$ is Euler's beta function.
The definition of $\la_{1,p}(\Om)$ and its scaling property give that
$$
\la_{1,p}(\Om)\le\la_{1,p}(B_{r_\Om})\le \frac{\la_{1,p}(B)}{r_\Om^p}.
$$
Therefore, by using Euler's gamma function, we conclude that
$$
\frac{ d_\Ga(z)}{r_\Om}\ge \frac1{p} \left[\frac{p-1}{\la_{1,p}(B)}\right]^{1/p} \Ga(1/p) \Ga(1/p').
$$
The constant at the right-hand side only depends on $N$ and $p$.
\end{rem}

\bigskip

\section{Anisotropic case: Wulff-type functionals}
\label{sec:Wulff}

Our analysis can be extended to a class of anisotropic problems. However, the proof of the corresponding Theorem \ref{th:maximum-point-torsion-quasilinear} needs some more detail. To avoid unnecessary complications, we shall present it for
the minimizer $u$ of the \textit{Wulff-type} functional
\begin{equation}
\label{Wulff-functional}
\int_\Om [\Phi(H(\na v))-N\,v]\,dx,
\end{equation}
among all the functions $v\in W^{1,p}_0(\Om)$ with $p>1$. Here, $H:\RR^N\to[0,\infty)$ is a suitable norm (see Section \ref{subsec:properties-H} for some definitions and relevant properties of $H$). 
\par
Whenever convenient, we will adopt the notation $\Phi_H=\Phi\circ H$ for short. 
The assumptions for $\Phi$ are those stated in Section \ref{subsec:quasilinear-setting}. In particular, we require that $\Phi_H$ satisfies \eqref{assumptions-Phi}.
\par
The strict convexity of the functional in \eqref{Wulff-functional} makes sure that a minimizer $u$ exists and is unique, and also satisfies the Dirichlet problem
\begin{equation}
\label{anisotropic-torsion-general-Phi}
-\dv \{\na\Phi_H(\na u)\}=N \ \mbox{ in } \ \Om, \quad u=0 \ \mbox{ on } \ \Ga,
\end{equation}
or, more explicitly, the problem
\begin{equation}
\label{anisotropicquasilinear-torsion-2}
-\dv\left\{ \phi( H( \na u ))\, \na_{ \xi} H( \na u) \right\}=N \ \mbox{ in } \ \Om, \quad u=0 \ \mbox{ on } \ \Ga,
\end{equation}
in the weak sense. 
\par
Before stating and proving the main results of this section, we need to recall some definitions, notations and relevant properties related to the norm $H$.

\subsection{Anisotropic norm, ball, curvature and distance}
\label{subsec:properties-H}

We assume that $H:\RR^N\to[0,\infty)$ is a norm on $\RR^N$, that is, it holds that
\begin{enumerate}[(i)]
\item $H( \xi) \ge 0$ for $\xi \in \RR^N$ and $H(\xi) = 0$ if and only if $\xi = 0$;
\item $H(t\,\xi) = |t|\,H(\xi)$ for $\xi \in \RR^N$ and $t \in \RR$;
\item $H$ satisfies the triangle inequality.
\end{enumerate}
Associated to $H$, we consider the dual norm on $\RR^N$ defined by the \textit{polar function}:
\begin{equation}
\label{def-H0}
H^o(\eta) = \sup_{\xi \neq 0 } \frac{\lan \xi, \eta\ran}{H( \xi )}, \ \eta \in \RR^N,
\end{equation} 
where the brackets denote the 
scalar product in $\RR^N$. Also, we have that
\begin{equation*}
H(\xi) = \sup_{\xi \neq 0 } \frac{\lan \eta, \xi\ran}{H^o(\eta)} \quad \text{for} \quad \xi \in \RR^N .
\end{equation*}
Thanks to \eqref{def-H0}, it holds that
\begin{equation}
\label{anisotropic-holder}
|\lan \xi , \eta\ran| \le H(\xi)\, H^o (\eta) \, , \quad \text{for any } \, \xi , \, \eta \in \RR^N .
\end{equation}
\par 
For convenience in this section, we shall drop the dependence on the norm $H$ in the relevant notation. For instance, we will simply denote by $B$ and $B^0$ the unit balls in the norms $H$ and $H^o$, that is we set
$$
B= \{ \xi \in \RR^N : H( \xi ) < 1 \} \ \mbox{ and } \ B^o= \{ \eta \in \RR^N : H^o(\eta) < 1 \} .
$$ 
Notice that $H$ and $H^o$ are nothing else than the support functions of $B^o$ and $B$, respectively (see \cite{CM} and \cite[Section 1.7]{Sc}). By the homogeneity of the norms $H$ and $H^o$, we can define the corresponding balls centered at a point $x\in\RR^N$ and with radius $r>0$:
\begin{eqnarray*}
&&B_r(x) = \{\xi \in \RR^N : H (\xi -x) < r \}=x+r\,B, \\
&&B^o_r(x) = \{\eta \in \RR^N : H (\eta -x) < r \}=x+r\,B^o.
\end{eqnarray*}
The sets $B_r^o (\eta)$ or $B_r(\xi)$ are also named {\it Wulff shapes} of $H$ or $H^o$.
\par
When $H \in C^1( \RR^ N \setminus \{ 0 \})$, the homogeneity property (ii) of the norm $H$ is equivalent to the so-called \textit{Euler's identity}:
\begin{equation}
\label{eq:proprietacheserveinconti}
\lan\na_{\xi} H (\xi ), \xi\ran = H(\xi) \, , \quad \text{for any } \, \xi \in \RR^N ,
\end{equation}
where the left-hand side is taken to be $0$ when $\xi = 0$. By the same homogeneity, we have that
\begin{equation}
\label{0-homogeneous}
\na_\xi H(t\,\xi)=\sgn(t)\,\na_\xi H(\xi) \ \mbox{ for  $\xi\ne 0$ and $t\ne 0$}.
\end{equation}
Later on we will also use the following properties (see \cite[Section 3.1]{CS} or \cite{BP, Xi}).
The identity
\begin{equation}\label{H-grad-Ho-identity}
H(\na_\eta H^o(\eta))=1 \ \mbox{ for } \ \eta\ne 0  
\end{equation}
holds true. Moreover, the map $H \, \na_\xi H$ is invertible and it holds that
$$
H \, \na_\xi H = \left(  H^0 \, \na_\eta H^0 \right)^{-1}.
$$
By \eqref{H-grad-Ho-identity}, \eqref{0-homogeneous}, and the homogeneity of $H$, the last formula is equivalent to
\begin{equation}\label{H-nuovaproprietaaggiuntaperconto}
H^0(\eta) \na_\xi H \left( \na_\eta H^0 \left( \eta \right) \right) = \eta . 
\end{equation}
\par
If we denote as usual by $\nu (x)$ the normal unit vector at a point $x\in\Ga$ pointing inward to $\Om$,
the corresponding \textit{anisotropic inner normal} $\nu_a(x)$ to $\Om$ is then defined by
$$
\nu_a (x)= \na_{\xi } H (\nu(x)) 
$$
If the solution $u$ of \eqref{anisotropic-torsion-general-Phi} is of class $C^1(\ol{\Om})$, being $\nu(x)=\na u(x)/|\na u(x)|$, we can infer that
\begin{equation}
\label{nu-a}
\nu_a(x)= \na_{\xi } H (\na u(x)),
\end{equation}
by the $0$-homogeneity of $\na_\xi H$. If $u\in C^2(\ol{\Om})$, then the \textit{anisotropic mean curvature} of $\Ga$ (with respect to the inner normal) is defined as
$$
\cM_a (x) = -\frac{1}{N-1} \dv[\nu_a (x)]=-\frac{1}{N-1} \dv\left[\na_\xi H (\na u(x))\right], \ x \in \Ga.
$$
We shall say that $\Ga$ is \textit{H-mean convex} if it is of class $C^2$ and  $\cM_a\ge 0$ on $\Ga$. 
\par
An identity analogous to \eqref{reilly-identity} also holds for the so-called \textit{anisotropic laplacian}
$$
\De^a u = \dv \left[ H(\na u) \na_\xi H \left( \na u \right) \right],
$$
that corresponds to the choice $\Phi(\si)=\si^2/2$ in \eqref{anisotropic-torsion-general-Phi}.  In fact, 
if we notice that for the first and second anisotropic normal derivatives we have that
$$
u_{\nu_a}= \lan \na u , \nu_a\ran = \lan\na u , \na_{\xi } H ( \na u )\ran = H(\na u)
$$
and 
\begin{equation}
\label{eq_ununuanisotropa}
u_{\nu_a \nu_a}= \lan (\na^2 u) \, \nu_a , \nu_a \ran = \lan (\na^2 u)\,\na_{\xi}H(\na u),\na_{\xi}H(\na u)\ran,
\end{equation}
we obtain the identity (see also\cite{Xi}):
\begin{equation}
\label{anisotropic-mean-curvature-identity}
\De^a u = u_{\nu_a \nu_a} - (N-1)  H(\na u)\, \cM_a \ \mbox{ on } \ \Ga.
\end{equation}
 
\par
The \textit{anisotropic distance} of $x \in \ol{\Om} $ to the boundary $\Ga$ is the function defined by
\begin{equation}\label{eq:distanzaanisotropaduale}
d_\Ga^o(x)= \min_{y \in \Ga} H^o (x-y),  \quad x \in \ol{\Om} .
\end{equation}
For more details on the anisotropic distance and, more in general, in Minkowski spaces we refer to \cite{CM} (for treatments in Finsler and Riemaniann geometries see also \cite{LN} and \cite{Sk}).
\par
In the following sections, we shall just use the fact that, by definition \eqref{eq:distanzaanisotropaduale}, for any $x \in \Om$ it holds that $B^o_r\subset \Om$ for $r=d^o_\Ga(x)$ and, being any anisotropic ball $B^o_r (x)$ a convex set, if $x_0$ is a point on $\Ga$ realizing the minimum in \eqref{eq:distanzaanisotropaduale}, the line segment joining $x_0$ to $x$ is contained in $\ol{B^o_r (x)} \subset \ol{\Om}$.
In particular, as noticed in \cite{CM}, due to the Minkowskian
structure of the space, $x_0$ is joined to $x$ by a
geodesic, which is a segment issuing from $x_0$ to $x$ that goes along
the anisotropic normal direction $\nu_a = \na_{\xi} H (  \nu (x_0) )$.

\subsection{The anisotropic torsional rigidity in a Wulff shape}
In the set up described in the previous section, the solution of \eqref{anisotropicquasilinear-torsion-2} in the Wulff shape $B_r^o(x)$ is easily computed.

\begin{lem}[Solution in the Wulff shape]
Let $H$ be a norm in $\RR^N$ such that $H\in C^1(\RR^N\setminus\{ 0\})$ and $B^o$ is
strictly convex. Let  $w^r : \ol{B^o_r(x)}\to [0,\infty)$ be the function  defined by
$$
w^r(y)=\Psi(r)-\Psi \left(H^o (y-x) \right) \ \mbox{ for } \ y\in \ol{B_r^o(x)},
$$
where $\Psi$ is as usual the Young's conjugate of $\Phi$. Then $w^r$ is of class $C^1(\RR^N)$ and is a weak solution of the problem \eqref{anisotropic-torsion-general-Phi}.
\end{lem}

\begin{proof}
We can always assume that $x=0$. It is clear that $w^r=0$ on $\pa B^o_r(x)$.  The $C^1(\RR^N)$-regularity of $w^r$ follows from our assumptions on $H$ and $\Phi$ and \cite[Lemma 3.1]{CS}. Moreover, we compute:
$$
\na w^r( \eta)=-\psi(H^o ( \eta))\,\na_\eta H^o( \eta) \ \mbox{ for } \ \eta \in B_r^o(x),
$$
where $\psi=\Psi'$.
Thus, \eqref{H-grad-Ho-identity} and \eqref{0-homogeneous} give that
\begin{multline*}
-\phi( H( \na w^r ))\, \na_{ \xi} H( \na w^r)=\phi(\psi(H^o ( \eta)))\,\na_\xi H \left( \na_\eta H^o( \eta) \right) =
\\ H^o ( \eta) \,\na_\xi H \left( \na_\eta H^o ( \eta) \right),
\end{multline*}
since $\phi$ and $\psi$ are inverse of one another.
Our claim follows from \eqref{H-nuovaproprietaaggiuntaperconto}.
\end{proof}

We can now derive an anisotropic version of Lemma \ref{lem:relationdist}. 
\begin{lem}[A bound from below]

\label{lem:anisotropic-relationdist-quasilinear}
Let $\Om\subset\RR^N$, $N\ge 2$, be a bounded domain with boundary $\Ga$. 
\par
Let $u\in C^0(\ol{\Om})\cap C^{1,\ga}(\Om)$, $0<\ga\le 1$, be the solution of \eqref{anisotropicquasilinear-torsion-2}. Then
\begin{equation}
\label{anisotropic-relationdist}
u(x)\ge \Psi( d^o_\Ga(x)) \ \mbox{ for every } x\in\ol{\Om}.
\end{equation}
\end{lem}

\begin{proof}
We proceed as usual. For $x\in\Om$ we let $r=d^o_\Ga(x)$ and consider the ball $B^o_r(x)$. Thus, we obtain the comparison (see, e.g., \cite[Lemma 4.2]{CRS}) $u\ge w^r$ on $B^o_r(x)$, and hence at $x$ we get:
$$
u(x)\ge w^r(x)=\Psi(r)-\Psi(0)=\Psi(d^o_\Ga(x)),
$$
as claimed.
\end{proof}

Next, we generalize \eqref{anisotropic-mean-curvature-identity} to the case of the operator in \eqref{anisotropicquasilinear-torsion-2}.

\begin{prop}
Let $H\in C^2(\RR^N\setminus\{ 0\})$ and let $v$ be a function of class $C^2$ and such that $\na v\ne 0$ in a neighborhood of $\Ga$. Then, it holds that 
$$
\dv\left\{ \phi( H( \na v ))\, \na_{ \xi} H( \na v) \right\} =- (N-1)\,\phi( H( \na v )) \cM_a +  \phi'(H(\na v)) v_{\nu_a \nu_a} \ \mbox{ on }  \ \Ga.
$$
\par 
In particular, if $u$ is the solution of \eqref{anisotropicquasilinear-torsion-2}, then
\begin{equation}
\label{identity-phi-H}
\phi'(H(\na u)) u_{\nu_H \nu_H} = -N + (N-1)\,\phi( H( \na u )) \cM_a \ \text {on }  \Ga.
\end{equation}
\end{prop}

\begin{proof}
By the Leibnitz formula for products, we have that 
\begin{multline*}
\dv\left\{ \phi( H( \na v ))\, \na_{ \xi} H( \na v) \right\} = \\
\frac{ \phi( H( \na v ))}{H(\na v)} \De^a v + H( \na v )\,\na\left[\frac{ \phi( H( \na v ))}{H(\na v)}\right]\cdot\na_\xi H(\na v) = 
\frac{ \phi( H( \na v ))}{H(\na v)} \De^a v + \\
\frac{H(\na v)\,\phi'(H(\na v))-\phi(H(\na v))}{H(\na v)}\,\lan (\na^2 v)\,\na_{\xi}H(\na v),\na_{\xi}H(\na v)\ran=
\\
\frac{ \phi( H( \na v ))}{H(\na v)} \left[\De^a v-v_{\nu_a \nu_a}\right] +  \phi'(H(\na v)) v_{\nu_a \nu_a}, 
\end{multline*}
where we have also used \eqref{eq_ununuanisotropa}.
Identity \eqref{anisotropic-mean-curvature-identity} then gives the first claim, and hence \eqref{identity-phi-H} follows at once.
\end{proof}

We now need a counterpart of Lemma \ref{lem:bound-gradient}. To avoid further difficulties, we shall limit our discussion to the case in which $\Ga$ is $H$-mean convex. We will use facts contained in \cite{CFV}. The statement of Lemma \ref{lem:anisotropic-bound-gradient-quasilinear} here below is slightly different from that contained in \cite{CFV}, because more naturally we chose to present the relevant $P$-function in terms of the Young conjugate.

\begin{lem}[Gradient estimate]
\label{lem:anisotropic-bound-gradient-quasilinear}
Let $\Om\subset\RR^N$, $N\ge 2$, be a bounded domain with H-mean convex boundary $\Ga$.
Let $u \in W^{1,p}(\Om) \cap L^\infty(\Om)$  be the weak solution of \eqref{anisotropicquasilinear-torsion-2}.
Then the function defined on $\ol{\Om}$ by
$$
P=\Psi \left( \phi \left( H(\na u ) \right) \right)+N\,\Bigl[u-\max_{\ol{\Om}} u\Bigr] 
$$ 
attains its maximum at some critical point of $u$. In particular, it holds that
$$
\Psi \left( \phi \left( H(\na u ) \right) \right) \le N\,\Bigl[\max_{\ol{\Om}} u-u\Bigr] \ \mbox{ on } \ \ol{\Om}.
$$
\end{lem}

\begin{proof}
The necessary regularity can be obtained as in the proof of Lemma \ref{lem:bound-gradient-quasilinear-torsion}. Next, by taking advantage of \cite[Proposition 4.1]{CFV} 
with $B= \Phi$ (the relevant regularity assumptions there assumed can be relaxed by an appropriate approximation argument), we have that the function $P$ satisfies the maximum principle away from the critical points of $u$ in $\Om$.
In other words, we have that the maximum of $P$ is attained either on $\Ga$ or at a
critical point of $u$ at which $P \le 0$. 
\par
Thus, we are left to prove that $P$ cannot attain its maximum on $\Ga$. Suppose that $P$ is not constant and attains its maximum at a point $\ol{x}\in \Ga$.
Notice that, by Hopf's lemma (see \cite{CT} or \cite{CRS}), $\Ga$ does not contain any critical point of $u$. Hence, we can apply Hopf's lemma to $P$ in some neighborhood of $\ol{x}$, and infer that $P_{\nu_a}(\ol{x})<0$.
\par
On the other hand, we compute that
\begin{multline*}
\na P= \Psi'(\phi(H(\na u)))\,\phi'(H(\na u))\, [\na^2 u] \na_\xi H(\na u) + N\,\na u = \\
H(\na u)\,\phi'(H(\na u)) [\na^2 u] \na_\xi H(\na u) + N\,\na u.
\end{multline*}
Hence,  \eqref{eq:proprietacheserveinconti}, \eqref{nu-a}, and \eqref{eq_ununuanisotropa} give that 
\begin{multline*}
P_{\nu_a}= H(\na u)\,\phi'(H(\na u)) \lan [\na^2 u] \na_\xi H(\na u), \na_\xi H(\na u)\ran + 
N\,\lan\na u, \na_\xi H (\na u)\ran = \\
 H(\na u) \left\lbrace \phi'(H(\na u)) u_{\nu_a \nu_a}  + N\right\rbrace = 
(N-1) \,H(\na u)\,\phi(H(\na u))\,\cM_a  \ge 0 \ \text{ on } \Ga,
\end{multline*}
being $\cM_a\ge 0$ on $\Ga$. Therefore, we reached a contradiction, which means that 
either the maximum of $P$ is attained at a critical point of $u$ or $P$ is constant on $\ol{\Om}$. In any case, we conclude that $P\le 0$ on $\ol{\Om}$.
\end{proof}

We define the {\it anisotropic inradius} $r^o_{\Om}$ by
\begin{equation}
r^o_{\Om}= \max_{\ol{\Om}}d^o_\Ga
\end{equation}
and call {\it anisotropic incenter} a point $x^o_{\Om}$ that attains the maximum.

\begin{proof}[Proof of Theorem \ref{th:dist-estimate-anisotropic}]
The proof runs similarly to that of Theorem \ref{th:maximum-point-torsion-quasilinear}.
We can apply Lemma \ref{lem:anisotropic-bound-gradient-quasilinear} and obtain that
\begin{equation}\label{eq:anisotropicboundgradientutile}
H(\na u) \le \psi(\Psi^{-1}(N\,[u(z)-u])) \ \mbox{ on } \ \ol{\Om}.
\end{equation}
By using the function $\chi$ defined in \eqref{def-chi}, if we take $x\in\Om$ and $y\in\Ga$ such that $d^o_\Ga(x)= H^o (x-y)$, we can compute:
\begin{multline*}
\chi(u(z))-\chi(u(z)-u(x))=\int_0^1 \frac{d}{dt} \chi(u(z)-u(x+t\,(y-x))\,dt= 
\\
\int_0^1  \frac{ \lan\na [u(z)- u(x+t\,(y-x))] , (y-x)\ran}{\psi(\Psi^{-1} (N\,[u(z)- u(x+t\,(y-x))]))}\,dt.
\end{multline*}
Next, we apply \eqref{anisotropic-holder} and infer that
\begin{multline*}
 \chi(u(z))-\chi(u(z)-u(x))\le
\\
 \int_0^1  \frac{ H \left( \na [u(x+t\,(y-x))] \right) }{\psi(\Psi^{-1} (N\,[u(z)- u(x+t\,(y-x))]))} H_0 (y-x) \,dt \le
H_0 (y-x)=d^o_\Ga(x),
\end{multline*}
where in the second inequality we used \eqref{eq:anisotropicboundgradientutile}.
Thus, choosing $x=z$ gives that $\chi(u(z)) \le d^o_\Ga(z)$.
\par
The rest of the proof runs as that of Theorem \ref{th:maximum-point-torsion-quasilinear}, provided $x_\Om$ and $d_\Ga(z)$ are replaced by $x^o_\Om$ and $d^o_\Ga(z)$, and Lemma \ref{lem:relationdist-quasilinear} is replaced by Lemma \ref{lem:anisotropic-relationdist-quasilinear}.
\end{proof}

\begin{rem}
It is clear that, repeating the arguments used in the proof of Corollary \ref{cor:bound-elasto-plastic} yields 
$$
\frac{d^o_\Ga(z)}{r^o_\Om}\ge \left(\frac{c}{N C}\right)^{1/p}.
$$
\end{rem}

\medskip

\section*{Acknowledgements}
The paper is partially supported by the Gruppo Nazionale per l'Analisi Matematica, la Probabilit\`a e le loro Applicazioni (GNAMPA) dell'Istituto Nazionale di Alta Matematica (INdAM). The second author is also supported by the Australian Research Council Discovery Project DP170104880
``N.E.W. Nonlocal Equations at Work''.

\bibliographystyle{plain}                              
%
%
\bibliography{References}

\begin{thebibliography}{10}

\bibitem{Al1}
G.~Alessandrini.
\newblock Critical points of solutions of elliptic equations in two variables.
\newblock {\em Ann. Scuola Norm. Sup. Pisa Cl. Sci. (4)}, 14(2):229--256
  (1988), 1987.

\bibitem{Al2}
G.~Alessandrini.
\newblock Critical points of solutions to the {$p$}-{L}aplace equation in
  dimension two.
\newblock {\em Boll. Un. Mat. Ital. A (7)}, 1(2):239--246, 1987.

\bibitem{ALR}
G.~Alessandrini, D.~Lupo, and E.~Rosset.
\newblock Local behavior and geometric properties of solutions to degenerate
  quasilinear elliptic equations in the plane.
\newblock {\em Appl. Anal.}, 50(3-4):191--215, 1993.

\bibitem{AM1}
G.~Alessandrini and R.~Magnanini.
\newblock The index of isolated critical points and solutions of elliptic
  equations in the plane.
\newblock {\em Ann. Scuola Norm. Sup. Pisa Cl. Sci. (4)}, 19(4):567--589, 1992.

\bibitem{AM2}
G.~Alessandrini and R.~Magnanini.
\newblock Elliptic equations in divergence form, geometric critical points of
  solutions, and {S}tekloff eigenfunctions.
\newblock {\em SIAM J. Math. Anal.}, 25(5):1259--1268, 1994.

\bibitem{MR2839047}
R.~Alvarado, D.~Brigham, V.~Maz'ya, M.~Mitrea, and E.~Ziad\'{e}.
\newblock On the regularity of domains satisfying a uniform hour-glass
  condition and a sharp version of the {H}opf-{O}leinik boundary point
  principle.
\newblock volume 176, pages 281--360. 2011.
\newblock Problems in mathematical analysis. No. 57.

\bibitem{BP}
G.~Bellettini and M.~Paolini.
\newblock Anisotropic motion by mean curvature in the context of {F}insler
  geometry.
\newblock {\em Hokkaido Math. J.}, 25(3):537--566, 1996.

\bibitem{BM2}
D.~Berti and R.~Magnanini.
\newblock Asymptotics for the resolvent equation associated to the
  game-theoretic {$p$}-laplacian.
\newblock {\em Appl. Anal.}, 98(10):1827--1842, 2019.

\bibitem{BM1}
D.~Berti and R.~Magnanini.
\newblock Short-time behavior for game-theoretic {$p$}-caloric functions.
\newblock {\em J. Math. Pures Appl. (9)}, 126:249--272, 2019.

\bibitem{BL}
H.~J. Brascamp and E.~H. Lieb.
\newblock On extensions of the {B}runn-{M}inkowski and {P}r\'{e}kopa-{L}eindler
  theorems, including inequalities for log concave functions, and with an
  application to the diffusion equation.
\newblock {\em J. Functional Analysis}, 22(4):366--389, 1976.

\bibitem{BDF}
L.~Brasco, G.~Dephilippis, and G.~Franzina.
\newblock Positive solutions to the sublinear {L}ane-{E}mden equation are
  isolated.
\newblock {\em Preprint arXiv 1911.09163}, 2019.

\bibitem{BM}
L.~Brasco and R.~Magnanini.
\newblock The heart of a convex body.
\newblock In {\em Geometric properties for parabolic and elliptic {PDE}'s},
  volume~2 of {\em Springer INdAM Ser.}, pages 49--66. Springer, Milan, 2013.

\bibitem{BMS}
L.~Brasco, R.~Magnanini, and P.~Salani.
\newblock The location of the hot spot in a grounded convex conductor.
\newblock {\em Indiana Univ. Math. J.}, 60(2):633--659, 2011.

\bibitem{CGS}
L.~Caffarelli, N.~Garofalo, and F.~Seg\`ala.
\newblock A gradient bound for entire solutions of quasi-linear equations and
  its consequences.
\newblock {\em Comm. Pure Appl. Math.}, 47(11):1457--1473, 1994.

\bibitem{CRS}
D.~Castorina, G.~Riey, and B.~Sciunzi.
\newblock Hopf {L}emma and regularity results for quasilinear anisotropic
  elliptic equations.
\newblock {\em Calc. Var. Partial Differential Equations}, 58(3):Paper No. 95,
  18, 2019.

\bibitem{CS}
A.~Cianchi and P.~Salani.
\newblock Overdetermined anisotropic elliptic problems.
\newblock {\em Math. Ann.}, 345(4):859--881, 2009.

\bibitem{CFV}
M.~Cozzi, A.~Farina, and E.~Valdinoci.
\newblock Gradient bounds and rigidity results for singular, degenerate,
  anisotropic partial differential equations.
\newblock {\em Comm. Math. Phys.}, 331(1):189--214, 2014.

\bibitem{CM}
G.~Crasta and A.~Malusa.
\newblock The distance function from the boundary in a {M}inkowski space.
\newblock {\em Trans. Amer. Math. Soc.}, 359(12):5725--5759, 2007.

\bibitem{CT}
M.~Cuesta and P.~Tak\'{a}\v{c}.
\newblock A strong comparison principle for positive solutions of degenerate
  elliptic equations.
\newblock {\em Differential Integral Equations}, 13(4-6):721--746, 2000.

\bibitem{DLT1}
H.~Deng, H.~Liu, and L.~Tian.
\newblock Critical points of solutions to a quasilinear elliptic equation with
  nonhomogeneous {D}irichlet boundary conditions.
\newblock {\em J. Differential Equations}, 265(9):4133--4157, 2018.

\bibitem{DLT2}
H.~Deng, H.~Liu, and L.~Tian.
\newblock Critical points of solutions for the mean curvature equation in
  strictly convex and nonconvex domains.
\newblock {\em Israel J. Math.}, 233(1):311--332, 2019.

\bibitem{Fr}
A.~Friedman.
\newblock {\em Partial differential equations of parabolic type}.
\newblock Courier Dover Publications, 2008.

\bibitem{Ga}
R.~J. Gardner.
\newblock The {B}runn-{M}inkowski inequality.
\newblock {\em Bull. Amer. Math. Soc. (N.S.)}, 39(3):355--405, 2002.

\bibitem{GJ}
D.~Grieser and D.~Jerison.
\newblock The size of the first eigenfunction of a convex planar domain.
\newblock {\em J. Amer. Math. Soc.}, 11(1):41--72, 1998.

\bibitem{Jo}
F.~John.
\newblock Extremum problems with inequalities as subsidiary conditions.
\newblock In {\em Studies and {E}ssays {P}resented to {R}. {C}ourant on his
  60th {B}irthday, {J}anuary 8, 1948}, pages 187--204. Interscience Publishers,
  Inc., New York, N. Y., 1948.

\bibitem{Ko}
N.~J. Korevaar.
\newblock Convex solutions to nonlinear elliptic and parabolic boundary value
  problems.
\newblock {\em Indiana Univ. Math. J.}, 32(4):603--614, 1983.

\bibitem{LN}
Y.~Li and L.~Nirenberg.
\newblock The distance function to the boundary, {F}insler geometry, and the
  singular set of viscosity solutions of some {H}amilton-{J}acobi equations.
\newblock {\em Comm. Pure Appl. Math.}, 58(1):85--146, 2005.

\bibitem{Li}
G.~M. Lieberman.
\newblock Boundary regularity for solutions of degenerate elliptic equations.
\newblock {\em Nonlinear Analysis: Theory, Methods \& Applications},
  12(11):1203--1219, 1988.

\bibitem{Ma}
R.~Magnanini.
\newblock An introduction to the study of critical points of solutions of
  elliptic and parabolic equations.
\newblock {\em Rend. Istit. Mat. Univ. Trieste}, 48:121--166, 2016.

\bibitem{MP2}
R.~Magnanini and G.~Poggesi.
\newblock Serrin's problem and {A}lexandrov's {S}oap {B}ubble {T}heorem:
  enhanced stability via integral identities.
\newblock {\em To appear in Indiana Univ. Math. Jour., preprint arXiv
  1708.07392}, 2017.

\bibitem{MP1}
R.~Magnanini and G.~Poggesi.
\newblock On the stability for {A}lexandrov's soap bubble theorem.
\newblock {\em J. Anal. Math.}, 139(1):179--205, 2019.

\bibitem{PP1}
L.~E. Payne and G.~A. Philippin.
\newblock Some applications of the maximum principle in the problem of
  torsional creep.
\newblock {\em SIAM J. Appl. Math.}, 33(3):446--455, 1977.

\bibitem{PP2}
L.~E. Payne and G.~A. Philippin.
\newblock Decay bounds for solutions of second order parabolic problems and
  their derivatives.
\newblock {\em Math. Models Methods Appl. Sci.}, 5(1):95--110, 1995.

\bibitem{PPV}
L.~E. Payne, G.~A. Philippin, and S.~Vernier~Piro.
\newblock Decay bounds for solutions of second order parabolic problems and
  their derivatives. {IV}.
\newblock {\em Appl. Anal.}, 85(1-3):293--302, 2006.

\bibitem{Sa}
S.~Sakaguchi.
\newblock Behavior of spatial critical points and zeros of solutions of
  diffusion equations.
\newblock In {\em Selected papers on differential equations and analysis},
  volume 215 of {\em Amer. Math. Soc. Transl. Ser. 2}, pages 15--31. Amer.
  Math. Soc., Providence, RI, 2005.

\bibitem{Sk}
T.~Sakai.
\newblock {\em Riemannian geometry}, volume 149 of {\em Translations of
  Mathematical Monographs}.
\newblock American Mathematical Society, Providence, RI, 1996.
\newblock Translated from the 1992 Japanese original by the author.

\bibitem{Sc}
R.~Schneider.
\newblock {\em Convex bodies: the {B}runn-{M}inkowski theory}, volume 151 of
  {\em Encyclopedia of Mathematics and its Applications}.
\newblock Cambridge University Press, Cambridge, expanded edition, 2014.

\bibitem{Se}
J.~Serrin.
\newblock A symmetry problem in potential theory.
\newblock {\em Arch. Rational Mech. Anal.}, 43:304--318, 1971.

\bibitem{To}
P.~Tolksdorf.
\newblock Regularity for a more general class of quasilinear elliptic
  equations.
\newblock {\em Journal of Differential equations}, 51(1):126--150, 1984.

\bibitem{Va}
S.~R.~S. Varadhan.
\newblock On the behavior of the fundamental solution of the heat equation with
  variable coefficients.
\newblock {\em Comm. Pure Appl. Math.}, 20:431--455, 1967.

\bibitem{Wa}
J.~L. Walsh.
\newblock {\em The {L}ocation of {C}ritical {P}oints of {A}nalytic and
  {H}armonic {F}unctions}.
\newblock American Mathematical Society Colloquium Publications, Vol. 34.
  American Mathematical Society, New York, N. Y., 1950.

\bibitem{Xi}
C.~Xia.
\newblock On a class of anisotropic problems.
\newblock {\em Ph.D. Thesis, Albert-Ludwigs-Universit\"at Freiburg im
  Breisgau}, 2012, online on \url{https://d-nb.info/1123470936/34}.

\end{thebibliography}

\end{document}